\RequirePackage{fix-cm}
\documentclass[smallextended]{svjour3}       
\smartqed  
\usepackage{graphicx}
\usepackage{geometry}
\usepackage{amsmath}
\usepackage{stmaryrd}
\usepackage{mathrsfs}
\usepackage{amssymb}
\usepackage{lipsum}
\usepackage{multirow}
\usepackage{amsfonts}
\usepackage{epstopdf}
\usepackage{algorithm}
\usepackage{algorithmic}
\usepackage{amsopn}
\usepackage{caption}
\usepackage{longtable}
\usepackage{url}
\usepackage{color}
\usepackage{cases}

\begin{document}

\title{A relaxation approach to UBPPs based on equivalent DC penalized factorized matrix programs
\thanks{This work is funded by the National Natural Science Foundation of China
 under project No.11571120.}
}


\author{Yitian Qian         \and
        Shaohua Pan 
}


\institute{Yitian Qian \at
          School of Mathematics, South China University of Technology, Guangzhou.
              \email{mayttqian@mail.scut.edu.cn}
           \and
           Shaohua Pan\at
           School of Mathematics, South China University of Technology, Guangzhou.
           \email{shhpan@scut.edu.cn}
}

\date{Received: date / Accepted: date}

\maketitle

\begin{abstract}
  This paper is concerned with the unconstrained binary polynomial program (UBPP),
  which has a host of applications in many science and engineering fields.
  By leveraging the global exact penalty for its DC constrained SDP reformulation,
  we achieve an equivalent DC penalized SDP, and propose a continuous relaxation
  approach by seeking the critical point of the Burer-Monteiro factorization
  for a finite number of DC penalized SDPs with increasing penalty factors.
  A globally convergent majorization-minimization (MM) method with extrapolation
  is also developed to capture such critical points. Under a mild condition,
  we show that the rank-one projection of the output for the relaxation approach
  is an approximate feasible solution of the UBPP and quantify the upper bound
  of its objective value from the optimal value. Numerical comparisons with
  the SDP relaxation method armed with a special random rounding technique
  and the DC relaxation approach based on the solution of linear SDPs confirm
  the efficiency of the proposed relaxation approach, which can solve the instance of \textbf{20000}
  variables in \textbf{15} minutes and yield an upper bound to the optimal value and
  the known best value with a relative error at most \textbf{1.824\%} and \textbf{2.870\%}, respectively.
\keywords{UBPP \and DC global exact penalty \and Burer-Monteiro factorization \and relaxation approach}
\subclass{90C27 \and 90C22 \and 90C26}
\end{abstract}

 \section{Introduction}\label{sec1}

 In this work, we are interested in the following unconstrained binary polynomial program:
 \begin{equation}\label{UBPP}
  \max_{x\in\{-1,1\}^N}\!\vartheta(x)
 \end{equation}
 where $\vartheta\!:\mathbb{R}^N\to\mathbb{R}$ is a polynomial function of degree $2d$.
 Such a problem has a host of applications; for example, the classical unconstrained
 binary quadratic program (UBQP), for which $\vartheta(x)\!=\!\langle x,Qx\rangle+c^{\mathbb{T}}x$
 for an $N\!\times N$ real symmetric matrix $Q$ and a vector $c\in\!\mathbb{R}^N$,
 is frequently used to formulate the optimization problems on graphs,
 facility locations problems, resources allocation problems, clustering problems,
 set partitioning problems, and various forms of assignment problems
 (see, e.g., \cite{Burer01,Kochenberger14,Luo01,Pardalos92,Phillips94}).
 A more general example is the problem \eqref{UBPP} with
 \begin{equation}\label{vtheta}
  \vartheta(x)=\prod_{i=1}^q\Big(x_i^{\mathbb{T}}Q_ix_i+c_i^{\mathbb{T}}x_i+a_i\Big),
 \end{equation}
 where $x=(x_1;x_2;\ldots;x_q)\in\mathbb{R}^{nq}$, and every $Q_i$ is an $n\times n$
 real symmetric matrix. For the application of the problem \eqref{UBPP} with such $\vartheta$,
 the interested reader is referred to the article \cite{He13}.

 Over the past several decades, many solution methods were developed for this class of
 NP-hard problems (see \cite{Kochenberger14} for the survey), which can be roughly
 classified into the exact method (see, e.g., \cite{Krislock14,Krislock17,Li12,Pham10}),
 the metaheuristic method (see, e.g., \cite{Wu15,Glover10}), and the continuous
 relaxation method (see, e.g., \cite{Anjos02,Chardaire94,Burer01,Goemans95}).
 Among the existing continuous relaxation methods, the semidefinite program (SDP) relaxation
 is the most popular one owing to the significant work \cite{Goemans95}, which states that
 for the max-cut and max 2-satisfiability problem, a special random rounding for the solution
 of the associated linear SDP problem can yield a feasible solution whose expected objective value
 is at least $0.87856$ times the optimal value. The later research works in this line
 mainly focus on the relaxation improvement by adding valid inequalities
 \cite{Helmberg98,Rendl10,Krislock14} or using the double nonnegative cone relaxation
 \cite{Kim16,Fu18}. This work is also concerned with the SDP relaxation,
 but its aim is to design a relaxation approach that can yield a desirable approximate
 feasible solution without any rounding technique.

 Notice that $z\in\!\{-1,1\}^N$ if and only if $Z\!=\!zz^{\mathbb{T}}$ is a rank-one PSD matrix
 with one diagonals. Then, it is not difficult to reformulate the UBPP \eqref{UBPP} as
 the following rank-one SDP:
 \begin{equation}\label{rank-BQP1}
  \min_{X\in\mathbb{S}^{p}}\Big\{f(X)\ \ {\rm s.t.}\ \
  {\rm rank}(X)\le 1,{\rm diag}(X)=e,X\in\mathbb{S}_{+}^{p}\Big\},
 \end{equation}
 where $p$ is a positive integer related to $N$ and $d$, $\mathbb{S}_{+}^{p}$
 denotes the set of all PSD matrices in $\mathbb{S}^p$, the set of
 all $p\times p$ real symmetric matrices, and $f\!:\mathbb{S}^p\to\mathbb{R}$ is
 a smooth function (determined by $\vartheta$) with gradient $\nabla\!f$ being
 Lipschitz continuous relative to $\mathbb{B}_{\Omega}$, a compact set containing
 $(1\!+\!\tau)\Omega-\tau\Omega$ for all $\tau\in[0,1]$ where
 $\Omega\!:=\!\{X\in\mathbb{S}_{+}^p\,|\,{\rm diag}(X)=e\}$. For example, by taking
 $v_d(x):=(1,x_1,\ldots,x_N,x^2_1,x_1x_2,\ldots,x_1x_N,\ldots,x_1^d,\ldots,x_N^d)^{\mathbb{T}}$,
 it is immediate to reformulate \eqref{UBPP} as \eqref{rank-BQP1} with
 $f(X)=\langle C_{\vartheta},X\rangle$ and $p=(\begin{smallmatrix} N+d\\d \end{smallmatrix})$
 for a $p\times p$ real symmetric matrix $C_{\vartheta}$. Of course, one can reformulate
 \eqref{UBPP} as \eqref{rank-BQP1} with a nonlinear $f$ but a small $p$ by
 the structure of $\vartheta$; see Section \ref{sec5.6}.

 Our relaxation approach is based on a global exact penalty for the difference
 of convex (DC) constrained SDP reformulation of the UBPP.
 Since ${\rm rank}(X)\le 1$ if and only if $\|X\|_*\!-\|X\|=0$,
 where $\|X\|_*$ and $\|X\|$ denote the nuclear norm and spectral norm of the matrix $X$,
 the problem \eqref{rank-BQP1} can be equivalently written as the DC constrained SDP problem:
 \begin{equation}\label{DC-BPP}
  \min_{X\in\mathbb{S}^p}\Big\{f(X)\ \ {\rm s.t.}\ \
  \langle I,X\rangle-\|X\|=0,\,{\rm diag}(X)=e,X\in\mathbb{S}_{+}^{p}\Big\}.
 \end{equation}
 The feasible set of \eqref{DC-BPP}, denoted by $\mathcal{F}$,
 is still combinatorial, which is the set of all $p\times p$ rank-one PSD binary matrices.
 Since numerically it is more difficult to handle DC constraints than to
 handle DC objective functions, we pay our attentions to its penalty problem
 \begin{equation}\label{epenalty}
  \min_{X\in\Omega}\Big\{f(X)+\rho(\langle I,X\rangle-\|X\|)\Big\}
 \end{equation}
 where $\rho>0$ is the penalty parameter. By \cite[Proposition 2.3\&Theorem 3.1]{BiPan16},
 the problem \eqref{epenalty} associated to every $\rho\ge\rho^*\!:=\!(1+2p)\alpha_{\!f}$
 has the same global optimal solution set as the original problem \eqref{DC-BPP} does,
 where $\alpha_{\!f}>0$ is the Lipschitz constant of the function $f$ on the set $\Omega$.

 Because it is almost impossible to achieve a global optimal solution of \eqref{epenalty}
 associated to a fixed $\rho\ge\rho^*$, the exact penalty methods based on model \eqref{epenalty}
 still need to solve many DC penalty problems in practice even if the threshold $\rho^*$ is known.
 Furthermore, the convex relaxation methods for the penalty problem \eqref{epenalty} all
 require an eigenvalue decomposition in each iterate, which forms the major computational
 bottleneck and restricts their scalability to large-scale problems. Inspired by the recent
 renewed interest in the Burer-Monteiro factorization method \cite{Burer01,Burer03}
 for low-rank matrix recovery (see, e.g., \cite{SunLuo16,Li18}),
 we consider the factorized form of \eqref{epenalty}:
 \begin{align}\label{BPP-factor}
   &\min_{V\in\mathbb{R}^{m\times p}}f(V^{\mathbb{T}}V)+\rho(\|V\|_F^2-\|V\|^2) \nonumber\\
   &\quad {\rm s.t.}\ \ V\!\in\mathcal{S}\!:=\big\{V\in\mathbb{R}^{m\times p}\,|\,\|V_j\|=1,j=1,\ldots,p\big\}
 \end{align}
 where $1<\!m<p$ is an appropriate integer and $V_j$ is the $j$th column of $V$.
 It is easy to verify that if $X^*$ is a global optimizer of rank $r$ for the problem
 \eqref{epenalty} associated to $\rho$, then $V^*\!=\sqrt{\Lambda^*}(P_I^*)^{\mathbb{T}}$
 is globally optimal to the problem \eqref{BPP-factor} associated to this $\rho$ and $m\ge r$,
 where $\Lambda^*$ is the diagonal matrix consisting of the first $m$ largest
 eigenvalues of $X^*$ and $P_{I}^*$ is the matrix consisting of the first $m$
 columns of the eigenvector matrix $P^*$ of $X^*$; and conversely, if $V^*$ is
 a global optimal solution of the problem \eqref{BPP-factor} associated to $\rho$,
 then $X^*=(V^*)^{\mathbb{T}}V^*$ is globally optimal to the problem \eqref{epenalty}
 with an additional constraint ${\rm rank}(X)\le m$. This means that in a global sense,
 the solution of \eqref{epenalty} can be replaced with the solution of its factorized
 form \eqref{BPP-factor}. In the rest of this work, we call problem \eqref{BPP-factor}
 a DC penalized matrix program though, to be exactly, it is not a DC program due to
 the nonconvex constraint $V\!\in\mathcal{S}$. For convenience, in the sequel we also use
 \begin{equation}\label{psi-fun}
   \widetilde{f}(V)\!:=f(V^{\mathbb{T}}V)\ \ {\rm and}\ \ \widetilde{\psi}(V)\!:=\psi(V^{\mathbb{T}}V)
   \ {\rm for}\ V\in\mathbb{R}^{m\times p}\ {\rm with}\ \psi(Z):=-\|Z\|.
 \end{equation}

 In Section \ref{sec4}, by seeking a finite number of approximate critical points
 for the DC penalized problem \eqref{BPP-factor} with increasing $\rho$,
 we propose a relaxation approach to the UBPP \eqref{UBPP}. The rank-one projection
 of its output is shown to be approximately feasible under a mild condition,
 and the upper bound for its objective value to the optimal value of \eqref{UBPP}
 is also quantified; see Theorem \ref{obj-bound2}. In particular,
 inspired by the recent works on the DC programs (see, e.g., \cite{Pham14,LeThi18,Pang17,LiuPong19}),
 we also propose an MM method with extrapolation to seek a critical point
 of a single penalty problem \eqref{BPP-factor}. This MM method does not
 belong to the DCA framework \cite{Pham97,LeThi18} due to the manifold constraint
 $V\!\in\mathcal{S}$, but the convergence is achieved for the whole sequence generated,
 whose limit lies in a critical point set smaller than that of common DC programs
 \cite{Pang17,LiuPong19}. Our convergence results generalize those of \cite{LiuPong19}
 to the setting where $f$ is allowed to be restricted on
 the manifold $\mathcal{S}$.

 As mentioned above, the SDP relaxations for combinatorial optimization have
 been studied intensively, but most of existing works focus on an upper bound
 for \eqref{UBPP} by solving the SDPs without the rank-one constraint,
 and if a lower bound is expected to obtain from their solutions,
 a tailored rounding technique such as \cite{Goemans95} is required.
 Our relaxation approach is based on the factorized form of a finite number
 of nonconvex SDPs and can provide an approximate feasible output with infeasibility
 lower than $5\times 10^{-9}$, for which the simple MATLAB common ``round''
 yields a feasible solution with the almost same objective value. We notice that
 Pham~Dinh and Le~Thi \cite{Pham10} ever proposed a relaxation approach for UBQPs
 based on the DC penalized problem for \eqref{UBPP} itself. However,
 our DC penalized problem comes from its equivalent DC constrained SDP \eqref{DC-BPP}.
 By comparing the results of Table \ref{table1} with those of \cite[Table 2\&4]{Pham10},
 one can see that the relaxation based on the DC penalized matrix programs is more effective.

 Recently, Jiang et al. \cite{Jiang21} proposed a relaxation approach to the quadratic
 assignment problem by solving a similar penalty problem for the corresponding rank-one
 double nonnegative SDP cone program. Unlike our relaxation approach, their method
 only solves a single penalty problem associated to a well-chosen penalty parameter.
 There indeed exists the best penalty parameter but to capture it is almost impossible,
 and if the chosen penalty parameter is greater than the unknown best one,
 the quality of solution will become worse. This implies that solving a single penalty problem
 will miss those outputs of high quality. For the UBQPs, Wen and Yin \cite{WenY13}
 ever provided a continuous approach by solving the factorization form of a linear SDP.
 Since their factorized form neglects the rank-one constraint,
 the obtained critical point is far from rank-one and can not
 provide a feasible solution to the UBQPs without rounding.

 To confirm the efficiency of our relaxation approach based on model \eqref{BPP-factor}
 (dcFAC for short), we compare its performance with that of SDPRR and dcSDPT3 for \textbf{119}
 Biq Mac Library instances with $100$ to $250$ variables. Among others,
 dcSDPT3 is the DC relaxation approach based on \eqref{epenalty} for which
 the involved linear SDP subproblems are solved with the software SDPT3 \cite{Toh99},
 and SDPRR is the SDP relaxation method armed with the random rounding
 technique in \cite{Goemans95} (see Section \ref{sec5} for the details).
 We also compare the performance of dcFAC with that of dcSDPT3 for \textbf{112}
 UBQPs from the G-set, OR-Library and Palubeckis instances with $800$ to $20000$
 variables, and with that of dcSNCG for \textbf{26} UBPP examples constructed with
 $\vartheta$ from \eqref{vtheta} for $q=2$. Here, dcSNCG is the DC relaxation approach
 based on \eqref{epenalty} for which the quadratic SDP subproblems are solved
 with the dual semismooth Newton method \cite{QiSun06} (see Section \ref{sec5}
 for the details). Numerical comparisons show that dcFAC is comparable to dcSDPT3
 in terms of the quality of the output if the latter uses the same adjusting
 rule of $\rho$ (only possible for $n\le 500$) as dcFAC does,
 otherwise dcFAC is superior to dcSDPT3. Moreover, dcFAC is significantly superior
 to SDPRR and dcSNCG by the quality of the outputs and the CUP time taken.
 For \textbf{119} Biq Mac Library instances, the outputs of dcFAC have
 the relative error with the optimal values at most \textbf{1.824\%} except
 the special \textbf{gka9b} and \textbf{gka10b}, and for \textbf{112} UBQP instances
 with $n\ge 800$, the relative gaps of its outputs from the known best values
 are at most \textbf{2.870\%}.
 \section{Notion and preliminaries}\label{sec2}

 Throughout this paper, $\mathbb{R}^{n_1\times n_2}$ represents the vector space of
 all $n_1\times n_2$ real matrices, equipped with the trace inner product
 $\langle \cdot,\cdot\rangle$ and its induced Frobenius norm $\|\cdot\|_F$,
 and $\mathbb{O}^{p}$ denotes the set of all $p\times p$ orthonormal matrices.
 For $X\in\mathbb{R}^{n_1\times n_2}$, $X_{\!J}$ with some index set $J\subseteq\{1,\ldots,n_2\}$
 means the submatrix of $X$ consisting of those columns $X_j$ with $j\in J$,
 $\|X\|$ and $\|X\|_*$ denote the spectral norm and nuclear norm of $X$, respectively,
 and $\mathbb{B}(X,\varepsilon)$ means the closed ball on Frobenius norm centered at $X$
 with radius $\varepsilon>0$. For every $X\in\mathbb{S}^p$,
 write $\lambda(X):=(\lambda_1(X),\ldots,\lambda_p(X))$ with $\lambda_1(X)\ge\cdots\ge\lambda_p(X)$
 and $\mathbb{O}(X)\!:=\{P\in\mathbb{O}^{p}\,|\,X=P{\rm Diag}(\lambda(X))P^{\mathbb{T}}\}$,
 and for every $P\in\mathbb{O}(X)$, let $P_{I}\in\mathbb{R}^{p\times m}$ denote the submatrix
 consisting of the first $m$ columns of $P$. Let $I$ and $e$ denote
 an identity matrix and a vector of all ones, whose dimensions are known from the context.
 For a closed set $\Delta\subseteq\mathbb{R}^{n_1\times n_2}$,
 $\delta_{\Delta}$ denotes the indicator function of the set $\Delta$, i.e., $\delta_{\Delta}(z)=0$
 if $z\in\Delta$, otherwise $\delta_{\Delta}(z)=+\infty$.
 Write $\mathcal{R}\!:=\{Z\in\mathbb{S}^{p}\,|\,{\rm rank}(Z)\le 1\}$.

 Now we recall from the monograph \cite{RW98} the notion of regular and (limiting) subdifferentials.
 \begin{definition}\label{Gsubdiff-def}
  Consider a function $h\!:\mathbb{R}^{n}\to(-\infty,+\infty]$ and a point
  $x\in\mathbb{R}^{n}$ with $h(x)$ finite.
  The regular subdifferential of $h$ at $x$, denoted by $\widehat{\partial}h(x)$,
  is defined as
  \[
    \widehat{\partial}h(x):=\bigg\{v\in\mathbb{R}^{n}\ \big|\
    \liminf_{x\ne x'\to x}
    \frac{h(x')-h(x)-\langle v,x'-x\rangle}{\|x'-x\|}\ge 0\bigg\};
  \]
  and the (basic) subdifferential (also known as the limiting subdifferential) of $h$ at $x$ is
  \[
    \partial h(x):=\Big\{v\in\mathbb{R}^{n}\ |\  \exists\,x^k\to x\ {\rm with}\ h(x^k)\to h(x),
    v^k\in\widehat{\partial}h(x^k)\ {\rm s.t.}\ v^k\to v\ {\rm as}\ k\to\infty\Big\}.
  \]
 \end{definition}
 \begin{remark}\label{remark-Fsubdiff}
  {\bf(a)} The sets $\widehat{\partial}h(x)$ and $\partial h(x)$
  are closed with $\widehat{\partial}h(x)\!\subseteq\partial h(x)$,
  and the former is also convex. When $h$ is convex,
  they reduce to the subdifferential of $h$ at $x$ in the convex analysis context.
  The vector $\overline{x}$ at which $0\in\partial h(\overline{x})$
  is called a critical point of $h$, and we denote by ${\rm crit}\,h$
  the critical point set of $h$.

  \noindent
  {\bf(b)} When $h$ is an indicator function of a closed set $\Delta$ in $\mathbb{R}^n$,
  $\widehat{\partial}h(x)$ and $\partial h(x)$ respectively reduce to the regular normal
  cone $\widehat{\mathcal{N}}_{\Delta}(x)$ and the normal cone $\mathcal{N}_{\Delta}(x)$.
 \end{remark}

 The following lemma characterizes the subdifferential of the concave function $\psi(Z)=-\|Z\|$.
 \begin{lemma}\label{minus-spectral}
  Fix any $X\!\in\mathbb{S}^{p}$ with $J_{\!X}\!:=\{j\in\{1,\ldots,p\}\,|\,|\lambda_j(X)|=\|X\|\}$.
  Then the subdifferential of $\psi$ defined in \eqref{psi-fun} at $X$ has the following expression:
  \[
   \partial\psi(X)=\left\{\begin{array}{cl}
               \big\{-{\rm sign}(\lambda_j(X))P_jP_j^{\mathbb{T}}\,|\ j\in J_{\!X},P\in\mathbb{O}(X)\big\}&{\rm if}\ X\ne 0,\\
               \big\{\{-e_je_j^{\mathbb{T}},e_je_j^{\mathbb{T}}\}\,|\ j\in J_{\!X},P\in\mathbb{O}(X)\big\}&{\rm if}\ X=0,
               \end{array}\right.
  \]
  where $e_j\in\mathbb{R}^p$ is the vector with the $j$th entry being $1$ and others being $0$;
  and when ${\rm rank}(X)=1$, $\psi$ is regular at $X$ with
  $\widehat{\partial}\psi(X)=\partial\psi(X)=-\partial(-\psi)(X)$.
 \end{lemma}
 \begin{proof}
 Let $h(z)\!:=-\|z\|_\infty$ for $z\in\mathbb{R}^p$. Notice that $\psi$ is the spectral function
 associated to $h$, i.e., $\psi(Z)=h(\lambda(Z))$ for any $Z\in\mathbb{S}^p$.
 By \cite[Theorem 6]{Lewis99}, we have
 \[
   \partial\psi(X)=\big\{P{\rm Diag}(\xi)P^{\mathbb{T}}\ |\
    P\in\mathbb{O}(X),\xi\in\partial h(\lambda(X))\big\}.
 \]
 In addition, by \cite[Corollary 9.21]{RW98} and the expression of $h$,
 it is easy to calculate that
 \[
   \partial h(\lambda(X))
   =\left\{\begin{array}{cl}
    \big\{-{\rm sign}(\lambda_j(X))e_j\,|\,j\in J_{\!X}\big\}&{\rm if}\ \lambda(X)\ne 0;\\
     \big\{\{e_j,-e_j\}\,|\,j\in J_{\!X}\big\}&{\rm if}\ \lambda(X)=0.
    \end{array}\right.
 \]
 From the last two equations, we obtain the first part. When ${\rm rank}(X)=1$,
 it is easy to check that $h$ is differentiable at $\lambda(X)$,
 and the result holds by \cite[Theorem 6]{Lewis99}.\qed
 \end{proof}

 \section{Stationary points of equivalent models}\label{sec3}

 As well known, when nonconvex models are equivalent in a global sense,
 they generally have different stationary point sets even local optimizer sets.
 Then, it is necessary to discuss the relations between the local optimizers
 of \eqref{epenalty} and those of \eqref{BPP-factor}, and so as their stationary points.
 \begin{proposition}\label{prop-local}
  The following statements hold for the local optimizers of problems \eqref{epenalty} and \eqref{BPP-factor}.
   \begin{itemize}
   \item [(i)] Every feasible point of \eqref{DC-BPP} is a local optimal solution,
               which is a strictly local optimizer of the problem \eqref{epenalty}
               associated to $\rho>\rho^*$, the threshold for the exact penalty problem \eqref{epenalty}.

   \item[(ii)] If $X^*$ is a local optimizer of rank $r$ for the problem \eqref{epenalty} associated to
               $\rho$, then $\sqrt{\Lambda^*}(P_I^*)^{\mathbb{T}}$ with $P^*\!\in\mathbb{O}(X^*)$
               and $\Lambda^*={\rm Diag}(\lambda_1(X^*),\ldots,\lambda_m(X^*))$
               is a local optimizer of \eqref{BPP-factor} associated to this $\rho$ and $m\ge r$.
               Conversely, if $V^*$ is a rank-one local optimizer of \eqref{BPP-factor}
               associated to $\rho>0$, then $(V^*)^{\mathbb{T}}V^*\in\mathcal{F}$
               and is a local optimal solution to \eqref{epenalty} associated to this $\rho$.
   \end{itemize}
  \end{proposition}
  \begin{proof}
  {\bf(i)} By the discreteness of $\mathcal{F}$, it is easy to verify that $\mathcal{F}$
  coincides with the local optimizer set of \eqref{DC-BPP}. Pick any $X\in\mathcal{F}$.
  By the proof of \cite[Theorem 3.1(b)]{BiPan16}, $X$ is a local optimizer of \eqref{epenalty}
  associated to $\rho\ge\rho^*$. So, there exists $\varepsilon\in(0,1)$ such that
  \[
    f(Z)+\rho^*(\langle I,Z\rangle-\|Z\|)\ge
    f(X)+\rho^*(\langle I,X\rangle-\|X\|)\quad\forall Z\in\mathbb{B}(X,\varepsilon).
  \]
  Since $\mathcal{F}$ is the set of all $p\times p$ rank-one PSD binary matrices,
  by reducing $\varepsilon$ if necessary, we have $\langle I,Z\rangle-\|Z\|>0$ for all $Z\in[\mathbb{B}(X,\varepsilon)\backslash\{X\}]\cap\Omega$. Together with the last
  inequality, for all $Z\in[\mathbb{B}(X,\varepsilon)\backslash\{X\}]\cap\Omega$
  and $\rho>\rho^*$, it holds that
  \[
   f(X)+\rho(\langle I,X\rangle-\|X\|)
   =f(X)+\rho^*(\langle I,X\rangle-\|X\|)<f(Z)+\rho(\langle I,Z\rangle-\|Z\|).
  \]
  {\bf(ii)} Let $X^*$ be a local optimizer of rank $r$ for \eqref{epenalty}.
  Then there exists $\varepsilon>0$ such that
  \begin{equation}\label{fineq}
    f(X)-\rho(\langle I,X\rangle-\|X\|)\ge f(X^*)-\rho(\langle I,X^*\rangle-\|X^*\|)
    \ \ {\rm for\ all}\ X\in\mathbb{B}(X^*,\varepsilon)\cap\Omega.
  \end{equation}
  From ${\rm diag}(X^*)=e$, each column of $V^*=\!\sqrt{\Lambda^*}(P_I^*)^{\mathbb{T}}$
  has a unit length, which implies that $V^*$ is feasible to \eqref{BPP-factor}.
  Take $\varepsilon'=\min(1,\frac{\varepsilon}{2(\|V^*\|+1)})$.
  For any $V\in\mathbb{B}(V^*,\varepsilon')\cap\mathcal{S}$,
  \[
    \|V^{\mathbb{T}}V\!-\!X^*\|_F
    =\|V^{\mathbb{T}}V\!-\!(V^*)^{\mathbb{T}}V^*\|_F
    \le\|V^{\mathbb{T}}V\!-\!V^{\mathbb{T}}V^*\|_F
      +\|V^{\mathbb{T}}V^*\!-\!(V^*)^{\mathbb{T}}V^*\|_F\le\varepsilon,
  \]
  which along with $V^{\mathbb{T}}V\in\Omega$ means that
  $V^{\mathbb{T}}V\in\mathbb{B}(X^*,\varepsilon)\cap\Omega$.
  Thus, from \eqref{fineq} we get
  \[
    f(V^{\mathbb{T}}V)+\rho(\|V\|_F^2-\|V\|^2)
    \ge f((V^*)^{\mathbb{T}}V^*)+\rho(\|V^*\|_F^2-\|V^*\|^2)
  \]
  for all $V\in\mathbb{B}(V^*,\varepsilon')\cap\mathcal{S}$.
  So, $V^*$ is a local optimizer of \eqref{BPP-factor}.
  The converse of part (ii) is easy to obtain by using part (i)
  and the feasibility of $(V^*)^{\mathbb{T}}V^*$ to \eqref{DC-BPP}. \qed
  \end{proof}
 \begin{definition}\label{Spoint-def1}
  We call $X\in\mathbb{S}^{p}$ a stationary point of \eqref{DC-BPP}
  if $0\in\nabla\!f(X)\!+\!\mathcal{N}_{\mathcal{F}}(X)$, and
  a stationary point of \eqref{epenalty} with $\rho>0$ if
  $0\in\nabla\!f(X)+\rho[I+\partial\psi(X)]+\mathcal{N}_{\Omega}(X)$;
  and call $V\in\mathbb{R}^{m\times p}$ a stationary point of \eqref{BPP-factor}
  with $\rho>0$ if
  $0\in \nabla\!\widetilde{f}(V)\!+\!\rho(2V+\partial\widetilde{\psi}(V))+\mathcal{N}_{\mathcal{S}}(V)$.
 \end{definition}
 \begin{remark}
  The stationary point in Definition \ref{Spoint-def1} for the DC problems \eqref{epenalty}
  and \eqref{BPP-factor} are stronger than the common one in the reference (see \cite{Pham97,Pham14,LiuPong19}),
  where $\partial\psi(X)$ and $\partial\widetilde{\psi}(V)$ are respectively replaced with
  their upper inclusions $-\partial(-\psi)(X)$ and $-\partial(-\widetilde{\psi})(V)$.
 \end{remark}
 \begin{proposition}\label{prop1-Spoint}
  Let $\widehat{\mathcal{F}}$ denote the stationary point set of the problem \eqref{DC-BPP},
  and let $\widehat{\Omega}_{\rho}$ and $\widehat{\mathcal{S}}_{\rho}$ denote
  the stationary point sets of the problems \eqref{epenalty} and \eqref{BPP-factor}
  associated to $\rho>0$. Then,
  \begin{description}
    \item [(i)] $\mathcal{F}=\widehat{\mathcal{F}}=\big\{X\in\mathbb{S}^p\,|\,0\in \nabla\!f(X)+\mathcal{N}_{\Omega}(X)+\mathcal{N}_{\mathcal{R}}(X)\big\}$.

   \item[(ii)] For any $\rho>0$, ${\rm crit}\Phi_{\rho}\subseteq\widehat{\Omega}_{\rho}$
               where $\Phi_{\rho}(Z)\!:=f(Z)+\rho(\langle I,Z\rangle\!-\!\|Z\|)+\delta_{\Omega}(Z)$,
               and every rank-one stationary point of the problem \eqref{epenalty} associated to $\rho>0$
               is a rank-one strictly local optimizer of \eqref{epenalty} associated to those $\rho>\rho^*$.

   \item[(iii)] For each $X\in\mathcal{F}$, there is a neighborhood in which
                the stationary points of \eqref{epenalty} associated to $\rho>\rho^*$
                are all rank-one if their objective values are not more than
               $\Phi_{\rho}(X)$.

   \item[(iv)] If $X\in\widehat{\Omega}_{\rho}$ has a rank $r\le m$,
               then $\sqrt{\Lambda}P_I^{\mathbb{T}}\!\in\widehat{\mathcal{S}}_{\!\rho}$
               where $\Lambda\!=\!{\rm Diag}(\lambda_1(X),\ldots,\lambda_m(X))$
               and $P\in\mathbb{O}(X)$; and conversely, if $V\in\widehat{\mathcal{S}}_{\rho}$ and there exists
               $(W,y)\in\partial\psi(V^{\mathbb{T}}V)\times\mathbb{R}^p$ such that
              $\nabla\!f(V^{\mathbb{T}}V)+\rho(I+W)+{\rm Diag}(y)\in\mathbb{S}_{+}^p$,
              then $V^{\mathbb{T}}V\in\widehat{\Omega}_{\rho}$.
  \end{description}
  \end{proposition}
  \begin{proof}
  {\bf(i)} Pick any $X\in\mathcal{F}$. Since $\mathcal{F}$
  coincides with the local optimizer set of \eqref{DC-BPP},
  we have $0\in\nabla\!f(X)+\mathcal{N}_{\mathcal{F}}(X)$,
  which means that $X\in\widehat{\mathcal{F}}$.
  Since $\mathcal{F}=\Omega\cap\mathcal{R}$,
  from \cite[Proposition 2.3]{BiPan16} and \cite[Section 3.1]{Ioffe08} we have
  $\mathcal{N}_{\mathcal{F}}(X)\subseteq\mathcal{N}_{\Omega}(X)+\mathcal{N}_{\mathcal{R}}(X)$.
  Since ${\rm rank}(X)=1$, from \cite[Proposition 3.6]{Luke13},
  $\widehat{\mathcal{N}}_{\mathcal{R}}(X)=\mathcal{N}_{\mathcal{R}}(X)$,
  which by \cite[Corollary 10.9]{RW98} and the convexity of $\Omega$
  implies that $\mathcal{N}_{\Omega}(X)+\mathcal{N}_{\mathcal{R}}(X)\subseteq
  \widehat{\mathcal{N}}_{\mathcal{F}}(X)\subseteq\mathcal{N}_{\mathcal{F}}(X)$.
  The two sides show that $\mathcal{N}_{\mathcal{F}}(X)
  =\mathcal{N}_{\Omega}(X)+\mathcal{N}_{\mathcal{R}}(X)$,
  so $0\in\nabla\!f(X)+\mathcal{N}_{\Omega}(X)+\mathcal{N}_{\mathcal{R}}(X)$.
  Thus, $\mathcal{F}\!=\widehat{\mathcal{F}}\subseteq\big\{X\!\in\mathbb{S}^p\,|\,0\!\in \nabla\!f(X)+\mathcal{N}_{\Omega}(X)+\mathcal{N}_{\mathcal{R}}(X)\big\}$.
  Notice that if $0\in\nabla\!f(X)+\mathcal{N}_{\Omega}(X)+\mathcal{N}_{\mathcal{R}}(X)$,
  then $X\in\Omega\cap\mathcal{R}=\mathcal{F}$. The second equality holds.

  \noindent
  {\bf(ii)} Pick any $X\in{\rm crit}\Phi_{\rho}$. By \cite[Exercise 10.10]{RW98}
  and the Lipschitz continuity of $\|\cdot\|$, we have
  \[
    \partial\Phi_{\rho}(X)\subseteq
    \nabla\!f(X)\!+\!\rho\big[I +\partial\psi(X)\big]+\mathcal{N}_{\Omega}(X),
  \]
  which implies that ${\rm crit}\Phi_{\rho}\subseteq\widehat{\Omega}_{\rho}$.
  Since every rank-one stationary point of \eqref{epenalty} lies in $\mathcal{F}$,
  by Proposition \ref{prop-local} (i), it is a rank-one strictly local optimizer
  of \eqref{epenalty} with $\rho>\rho^*$.

  \noindent
  {\bf(iii)} By the proof of \cite[Theorem 3.1(b)]{BiPan16}, every $X\in\mathcal{F}$ is
  a local optimizer of \eqref{epenalty} with $\rho\ge\rho^*$.
  Then, there exists $\varepsilon>0$ such that $\Phi_{\rho^*}(Z)\ge \Phi_{\rho^*}(X)$
  for all $Z\in\mathbb{B}(X,\varepsilon)\cap\Omega$. Fix any $\rho>\rho^*$. Pick
  any $X_{\rho}\in\widehat{\Omega}_{\rho}\cap\mathbb{B}(X,\varepsilon)$. Then,
  from the given assumption it follows that
  \begin{align*}
   f(X_{\rho})+\rho(\langle I,X_{\rho}\rangle-\|X_{\rho}\|)
   &\le\Phi_{\rho}(X)=f(X)=\Phi_{\rho^*}(X)\le\Phi_{\rho^*}(X_{\rho})\\
   &=f(X_{\rho})+\rho^*(\langle I,X_{\rho}\rangle-\|X_{\rho}\|),
  \end{align*}
  which by $\rho>\rho^*$ implies that $\langle I,X_{\rho}\rangle-\|X_{\rho}\|=0$.
  Together with $X_{\rho}\in\mathbb{S}_{+}^p$, it follows that $\lambda_2(X_{\rho})
  =\cdots=\lambda_{p}(X_{\rho})=0$. Hence, the matrix $X_{\rho}$ is rank-one.

  \noindent
  {\bf(iv)} Fix any $V\!\in\mathcal{S}$. By \cite[Theorem 10.6]{RW98},
  we have $\partial\widetilde{\psi}(V)=2V\partial\psi(V^{\mathbb{T}}V)$.
  Since $\widehat{\mathcal{N}}_{\mathcal{S}}(V)=\big\{V{\rm Diag}(w)\,|\, \,w\in\mathbb{R}^p\big\}
    =\mathcal{N}_{\mathcal{S}}(V)$, the set $\mathcal{S}$ is Clarke regular.
  Thus, $V$ is a sttionary point of \eqref{BPP-factor} if and only if
  there exist $W\!\in\partial\psi(V^{\mathbb{T}}V)$ and $y\in\mathbb{R}^{p}$ such that
  \begin{equation}\label{crit-equaV}
   V\big[\nabla\!f(V^{\mathbb{T}}V)\!+\!\rho(I+W)+{\rm Diag}(y)\big]=0.
  \end{equation}
  Now pick any $X\!\in\widehat{\Omega}_{\rho}$ with ${\rm rank}(X)\!\le m$.
  Clearly, $X\in\Omega$. By \cite[Theorem 6.14]{RW98}, we have
  \(
    \mathcal{N}_{\Omega}(X)\!=\{{\rm Diag}(z)\,|\, z\in\mathbb{R}^p\}
     +\mathcal{N}_{\mathbb{S}_{+}^p}(X).
  \)
  From Definition \ref{Spoint-def1}, there exist $W\in\partial\psi(X),
  y\in\mathbb{R}^p$ and $S\in\mathcal{N}_{\mathbb{S}_{+}^p}(X)$
  such that
  \(
    0=\nabla\!f(X)+\rho(I+W)+{\rm Diag}(y)+S.
  \)
  Let $V=\sqrt{\Lambda}P_I^{\mathbb{T}}$ with $P=\mathbb{O}(X)$
  and $\Lambda={\rm Diag}(\lambda_1(X),\ldots,\lambda_m(X))$.
  Notice that $VS=0$. Then
  \(
    V\big[\nabla\!f(X)+\rho(I+W)+{\rm Diag}(y)\big]=0.
  \)
 Since $X=V^{\mathbb{T}}V$, from \eqref{crit-equaV} we have
 $V\in\widehat{\mathcal{S}}_{\!\rho}$.
 For the second part, by taking $X=V^{\mathbb{T}}V$
 and $S=\nabla\!f(V^{\mathbb{T}}V)+\rho(I+W)+{\rm Diag}(y)\in\mathbb{S}_{+}^p$,
 from \eqref{crit-equaV} we obtain $XS=0$. Hence,
 $-S\in\mathcal{N}_{\mathbb{S}_{+}^p}(X)$, and
 $X\in\widehat{\Omega}_{\rho}$ follows by Definition \ref{Spoint-def1}.
 The proof is completed. \qed
 \end{proof}
 \begin{remark}\label{remark1-Spoint}
  By Proposition \ref{prop1-Spoint} (i), for every $X\in\mathcal{F}$,
  there exists a matrix $H\in\mathcal{N}_{\Omega}(X)$ such that
  $-\rho^{-1}(\nabla\!f(X)\!+H)\in\mathcal{N}_{\mathcal{R}}(X)$,
  but $-\rho^{-1}(\nabla\!f(X)+H)$ may not belong to $I+\partial\psi(X)$
  which is the singleton $\{I\!-X/\|X\|\}$ by Lemma \ref{minus-spectral}.
  This means that the rank-one stationary point set of \eqref{epenalty}
  associated to any $\rho>0$ is far smaller than $\mathcal{F}$,
  and from the last part of Proposition \ref{prop1-Spoint} (ii),
  it is also a rank-one strictly local optimizer set when $\rho>\rho^*$.
 \end{remark}
 \section{Relaxation approach based on model \eqref{BPP-factor}}\label{sec4}

  Inspired by the relationship between \eqref{epenalty} and \eqref{BPP-factor},
  we propose the following continuous relaxation approach by seeking a finite
  number of critical points of \eqref{BPP-factor} associated to increasing $\rho$.
  \begin{algorithm}[H]
  \caption{(DC relaxation approach based on \eqref{BPP-factor})}
  \label{Alg-factor}
  \begin{algorithmic}
  \normalsize
  \STATE{Select an integer $m>1$, a small $\epsilon\in(0,1)$, and appropriately large
  $l_{\rm max}\in\mathbb{N}$ and $\rho_{\rm max}\!>0$. Choose $\rho_0>0,\sigma>1$
  and a starting point $V^0\in\mathcal{S}$.}
  \FOR{$l=0,1,2,\ldots,l_{\rm max}$}
  \STATE{Starting from $V^{l}\in\mathcal{S}$, seek a critical point $V^{l+1}$ of the nonconvex problem
             \begin{equation}\label{epenalty-factor}
             \min_{V\in\mathcal{S}}\big\{\widetilde{f}(V)+\rho_{l}[\|V\|_F^2+\widetilde{\psi}(V)]\big\}.
             \end{equation}}
  \STATE{If $\|V^{l+1}\|_F^2-\|V^{l+1}\|^2\le\epsilon$, then stop. Otherwise, $\rho_{l+1}\leftarrow\min\{\sigma\rho_l,\rho_{\rm max}\}$.}
  \ENDFOR
  \end{algorithmic}
  \end{algorithm}
  The core of Algorithm \ref{Alg-factor} is to achieve a stationary point
  of \eqref{epenalty-factor} efficiently. Notice that the function $\widetilde{f}$
  is smooth with gradient $\nabla\!\widetilde{f}$ being Lipschitz continuous relative to $\mathbb{B}_{\mathcal{S}}$,
  a compact set containing $(1+\tau)\mathcal{S}-\tau\mathcal{S}$ for all $\tau\in[0,1]$.
  We denote by $L_{\!\widetilde{f}}$ the Lipschitz constant of $\nabla\!\widetilde{f}$
  relative to $\mathbb{B}_{\mathcal{S}}$. Fix any $Z\in\mathbb{B}_{\mathcal{S}}$.
  From the descent lemma, it follows that for any $V\in\mathbb{B}_{\mathcal{S}}$,
  \begin{subequations}
   \begin{align}
   \label{wfrho}
   \widetilde{f}(V)&\le\widetilde{f}(Z)+\langle\nabla\!\widetilde{f}(Z),V\!-\!Z\rangle
   +(L_{\!\widetilde{f}}/2)\|V\!-\!Z\|_F^2,\\
   \label{-wfrho}
   -\widetilde{f}(V)&\le -\widetilde{f}(Z)-\langle\nabla\!\widetilde{f}(Z),V\!-\!Z\rangle
   +(L_{\!\widetilde{f}}/2)\|V\!-\!Z\|_F^2.
   \end{align}
  \end{subequations}
  Notice that $\widetilde{\psi}$ is concave since $\widetilde{\psi}(V)=-\|V^{\mathbb{T}}V\|=-\|V\|^2$
  for any $V\in\mathbb{R}^{m\times p}$. Hence, $\widetilde{\psi}(V)\le\widetilde{\psi}(Z)+\langle\Gamma, V\!-\!Z\rangle$
  for any $\Gamma\in\partial\widetilde{\psi}(Z)$. Together with \eqref{wfrho}, we have
  \begin{align*}
   \widetilde{f}(V)+\rho_l\big(\|V\|_F^2+\widetilde{\psi}(V)\big)
   &\le\widetilde{F}(V,Z):=\langle\nabla\!\widetilde{f}(Z)\!+\rho_l\Gamma,V\rangle+\rho_l\|V\|_F^2
                          +\frac{L_{\!\widetilde{f}}}{2}\|V\!-\!Z\|_F^2 \\
   &\qquad\qquad\qquad +\widetilde{f}(Z)+\rho_l\widetilde{\psi}(Z)-\langle\nabla\!\widetilde{f}(Z),Z\rangle
   -\rho_l\langle\Gamma,Z\rangle.
  \end{align*}
  Along with $\widetilde{F}(Z,Z)=\widetilde{f}(Z)+\rho_l(\|Z\|_F^2+\widetilde{\psi}(Z))$,
  $\widetilde{F}(\cdot,Z)$ is a majorization of the cost function of \eqref{epenalty-factor} at $Z$.
  By this, we propose an MM method with extrapolation, which is not affiliated to
  the DCA \cite{Pham97} due to $\Gamma^{k}\in\partial\widetilde{\psi}(V^k)$
  and the manifold constraint $V\in\mathcal{S}$.
 \begin{algorithm}
 \renewcommand{\thealgorithm}{A}
 \caption{(MM method with extrapolation for \eqref{epenalty-factor})}
 \label{Alg2}
 \begin{algorithmic}
 \normalsize
 \STATE{Fix an integer $l\ge 0$. Choose $0\le\!\beta_0\!\le\!\overline{\beta}<1$
        and $L_0\ge\underline{L}>L_{\!\widetilde{f}}$. Set $\rho=\rho_l$ and $V^{-1}\!=V^0\!=V^l$.}
 \FOR{$k=0,1,2,\ldots$}
 \STATE{Choose an element $\Gamma^{k}\in\partial\widetilde{\psi}(V^k)$. Let $U^k=V^k+\beta_k(V^k\!-\!V^{k-1})$ and compute
        \begin{equation}\label{Vk-subprob}
         \!V^{k+1}\in\mathop{\arg\min}_{V\in\mathcal{S}}
            \Big\{\langle\nabla\!\widetilde{f}(U^k)+\rho \Gamma^{k},V\rangle
             \!+\!\rho\|V\|_F^2\!+\!\frac{L_k}{2}\|V\!-\!U^k\|_F^2\Big\}.\qquad
       \end{equation}}
 \STATE{Update $\beta_{k}$ by $\beta_{k+1}\in[0,\overline{\beta}]$ and $L_k$ by $L_{k+1}\in[\underline{L},L_0]$}.
 \ENDFOR
\end{algorithmic}
\end{algorithm}
 \begin{remark}\label{remark-Alg2}
  {\bf(a)} Since $L_{\!\widetilde{f}}$ may be unknown in practice,
  one can search a suitable $L_k$ by the descent lemma.
  When $L_{\!\widetilde{f}}$ is known, it suffices to choose
  $L_k\equiv(1+\delta)L_{\!\widetilde{f}}$ for a tiny $\delta>0$. As will be shown below,
  the restriction $L_0>L_{\!\widetilde{f}}$ is necessary for the global
  convergence of Algorithm \ref{Alg2} due to the nonconvexity of \eqref{Vk-subprob}.

  \noindent
  {\bf(b)} By the proof of Proposition \ref{prop1-Spoint} (iv),
  $\partial\widetilde{\psi}(V^k)\!=2V^k\partial\psi(X^k)$
  with $X^k\!=(V^k)^{\mathbb{T}}V^k$. Thus, by Lemma \ref{minus-spectral},
  one can choose $\Gamma^k=-2V^kP_1^k(P_1^k)^{\mathbb{T}}$ with $P^k\in\mathbb{O}(X^k)$.
  Clearly, $P_1^k$ can be achieved by the SVD of $V^k$, whose computation cost
  is cheaper since $V^k$ has less rows. We stipulate that $\Gamma^k$ in
  Algorithm \ref{Alg2} is always chosen in this way.

  \noindent
  {\bf(c)} Let $G^k\!:=\frac{1}{L_k+2\rho}(L_kU^k\!+\!\rho\Gamma^k\!-\!\nabla\!\widetilde{f}(U^k))$.
  Write $J_k:=\{j\,|\,\|G_j^k\|\ne 0\}$. Then $V^{k+1}$ with $V_j^{k+1}\!=\frac{G_j^k}{\|G_j^k\|}$
  for $j\in J_k$ and $V_j^{k+1}=(1,0,\ldots,0)^{\mathbb{T}}\in\mathbb{R}^m$ for $j\notin J_k$
  is an optimal solution of \eqref{Vk-subprob}. So, the computation cost in
  each step of Algorithm \ref{Alg2} is very cheap.
 \end{remark}

 In order to establish the convergence of Algorithm \ref{Alg2},
 we define the potential function
 \begin{equation}\label{Theta-rho}
  \Theta_{\rho}(V,\Gamma,U)\!:=\!\widetilde{f}(V)\!+\!\rho\|V\|_F^2
  +\rho\langle\Gamma,V\rangle+\rho(-\widetilde{\psi})^*(-\Gamma)+\delta_{\mathcal{S}}(V)
  +\frac{\gamma\underline{L}}{2}\|V\!-\!U\|_F^2
 \end{equation}
 for $(V,\Gamma,U)\in\mathbb{R}^{m\times p}\times\mathbb{R}^{m\times p}\times\mathbb{R}^{m\times p}$,
 where $\gamma\in\big(0,\frac{\underline{L}-L_{\!\widetilde{f}}}{2\underline{L}}\big)$ is a constant.
 The following proposition states the properties of the sequence $\{(V^k,\Gamma^k)\}$,
 whose proof is included in Appendix A.
 \begin{proposition}\label{prop-Alg2}
  Let $\{(V^k,\Gamma^k)\}$ be the sequence given by Algorithm \ref{Alg2}. The following results hold.
 \begin{itemize}
  \item[(i)] For each $k\in\mathbb{N}\cup\{0\}$, with
             $\nu_k:=\frac{(\gamma\underline{L}-2L_{\!\widetilde{f}}\beta_k^2)(L_k-L_{\!\widetilde{f}}-\gamma \underline{L})-(L_k-L_{\!\widetilde{f}})^2\beta_k^2}{L_k-L_{\!\widetilde{f}}-\gamma\underline{L}}$,
             it holds that
             \[
               \Theta_{\rho}(V^{k+1},\Gamma^{k},V^k)
               \le\Theta_{\rho}(V^{k},\Gamma^{k-1},V^{k-1})
               -\frac{\nu_k}{2}\|V^{k}\!-\!V^{k-1}\|_F^2;
                 \vspace{-0.3cm}
             \]

  \item [(ii)] The sequence $\{(V^k,\Gamma^k)\}$ is bounded, and hence
               the accumulation point set of the sequence
                $\{(V^k,\Gamma^{k-1},V^{k-1})\}$, denoted by $\Delta_{\rho}$, is nonempty and compact;

  \item[(iii)] When $\overline{\beta}<\!\sqrt{\frac{\gamma\underline{L}(L_0-L_{\!\widetilde{f}}
               -\gamma\underline{L})}{L_0^2-2\gamma\underline{L}L_{\!\widetilde{f}}-L_{\!\widetilde{f}}^2}}$,
               the limit $\lim_{k\to\infty}\Theta_{\rho}(V^k,\Gamma^{k-1},V^{k-1})$ exists,
               and moreover, the function $\Theta_{\rho}$ keeps unchanged on the set $\Delta_{\rho}$;

  \item[(iv)] For all $k\in\mathbb{N}$, the following inequality holds with
              $\alpha=\sqrt{2(L_{\widetilde{f}}+L_{0}+\gamma\underline{L})^2+\rho^2+\gamma^2\underline{L}^2}$:
              \[
                {\rm dist}(0,\partial\Theta_{\rho}(V^k,\Gamma^{k-1},V^{k-1}))
               \!\le\alpha\big[\|V^k\!-\!V^{k-1}\|_F+\|V^{k-1}\!-\!V^{k-2}\|_F\big].
              \]
 \end{itemize}
 \end{proposition}
 \begin{remark}\label{remark-VkGamk}
  {\bf(a)} Let $L_0=\kappa L_{\!\widetilde{f}}$ for $\kappa>1$ and
  $\underline{L}=cL_{\!\widetilde{f}}$ for $1<c\le\kappa$. If $\gamma>0$
  is such that $2\gamma c\le c-1$, then
  \(
    \!\sqrt{\frac{\gamma\underline{L}(L_0-L_{\!\widetilde{f}}
    -\gamma\underline{L})}{L_0^2-2\gamma\underline{L}L_{\!\widetilde{f}}-L_{\!\widetilde{f}}^2}}
    =\sqrt{\frac{\gamma c(\kappa-1-\gamma c)}{\kappa^2-2\gamma c-1}}.
  \)
  Now if $\widetilde{f}$ is convex, the restriction on $\overline{\beta}$
  is updated to $\overline{\beta}<\sqrt{\frac{\gamma c(\kappa-1-\gamma c)}{\kappa^2-\gamma c-\kappa}}$
  because the coefficient $L_{\widetilde{f}}$ in the first term
  of \eqref{WL-Vk} can be removed.

  \noindent
  {\bf(b)} By Remark \ref{remark-Fsubdiff} and the proof of part (iv),
  we have $\Delta_{\rho}\subseteq{\rm crit}\Theta_{\rho}$. While from \eqref{subdiff-Thetarho}
  and Definition \ref{Spoint-def1}, one can check that
  $\Pi_1({\rm crit}\Theta_{\rho})\subseteq\widehat{\mathcal{S}}_{\rho}$,
  where $\Pi_1(V,\Gamma,U)=V$ for $(V,\Gamma,U)\in\mathbb{R}^{m\times p}
  \times\mathbb{R}^{m\times p}\times\mathbb{R}^{m\times p}$.
  The two sides imply that $\Pi_1(\Delta_{\rho})\subseteq\Pi_1({\rm crit}\Theta_{\rho})
  \subseteq\widehat{\mathcal{S}}_{\rho}$.
 \end{remark}

  By \cite[Proposition 11.21]{RW98}, $(-\widetilde{\psi})^*(U)=\frac{1}{4}\|U\|_*^2$
  for $U\in\mathbb{R}^{m\times p}$. Clearly, $(-\widetilde{\psi})^*$ is semialgebraic.
  Since $\delta_{\mathcal{S}}$ and $(-\widetilde{\psi})^*$ are semialgebraic,
  $\Theta_{\rho}$ is semialgebraic and is a KL function. By Proposition \ref{prop-Alg2}
  and Remark \ref{remark-VkGamk} (b), using the same arguments as those for \cite[Theorem 3.2]{Attouch10}
  or \cite[Theorem 3.1]{LiuPong19} yields the following convergence theorem.
 \begin{theorem}\label{theorem-Alg2}
  Let $\{(V^k,\Gamma^k)\}$ be the sequence generated by Algorithm \ref{Alg2} for solving
  \eqref{BPP-factor} associated to $\rho$ with $\overline{\beta}$
  satisfying the restriction in Proposition \ref{prop-Alg2} (iii). Then,
  $\{V^k\}$ is convergent and its limit $V^*$ is a stationary point of the problem \eqref{BPP-factor}
  associated to $\rho$. If the limit $V^*$ is rank-one, then $(V^*)^{\mathbb{T}}V^*$
  is a local optimal solution of the problem \eqref{DC-BPP}.
 \end{theorem}

  Next we focus on the stopping criterion of Algorithm \ref{Alg-factor}. In the sequel, we say that
  Algorithm \ref{Alg-factor} exits normally if it stops at some $l<l_{\rm max}$.
  To show that Algorithm \ref{Alg-factor} armed with Algorithm \ref{Alg2} can exit normally,
  we need the following technical lemma, which states that if there exists an eigenvector
  associated to $\lambda_{1}((V^{l})^{\mathbb{T}}V^{l})$ having no zero entries,
  the gap $\|V^{l,1}\|_F^2-\|V^{l,1}\|^2$ is small.
 \begin{lemma}\label{complexity}
  Fix an integer $l\ge 0$. Suppose that there exists $P\in\mathbb{O}((V^{l})^{\mathbb{T}}V^{l})$
  such that $P_1$ has no zero entries. Let $|P_{1\kappa}|:=\min_{1\le j\le p}|P_{1j}|$ and
  $\varpi:=(L_0\!+2L_{\!\widetilde{f}})\sqrt{p}+\|\nabla\!\widetilde{f}(\frac{1}{\sqrt{m}}E)\|_F$.
  Then, when $\rho_l>\overline{\rho}:=\frac{L_0p+\varpi}{|P_{1\kappa}|c_0}$
  for some $c_0\in(0,1)$, $\|V^{l,1}\|_F^2-\|V^{l,1}\|^2\le c_0$.
 \end{lemma}
 \begin{proof}
  Recall that $\Gamma^0=-2V^{0}P_1P_1^{\mathbb{T}}$ by Remark \ref{remark-Alg2} (b).
  After a simple calculation, for each $j\in\{1,2,\ldots,p\}$,
  we have $\|\Gamma_{\!j}^0\|=2\|V^0\||P_{1j}|\ge 2|P_{1j}|\ge2|P_{1\kappa}|>0$,
  where the first inequality is due to $\|V^0\|\ge 1$ implied by $\|V^0\|_F=\sqrt{p}$.
  Notice that $V^0\in\mathcal{S}$ and
  $\|\nabla\!\widetilde{f}(U^0)\|_F\le\!L_{\!\widetilde{f}}\|V^0\!-\!\frac{1}{\sqrt{m}}E\|_{F}
  +\|\nabla\!\widetilde{f}(\frac{1}{\sqrt{m}}E)\|_F$.
  It is not hard to verify that
  $\|L_0U^0\!-\!\nabla\!\widetilde{f}(U^0)\|_F\le
  L_0\|V^0\|_F+\|\nabla\!\widetilde{f}(U^0)\|_F\le\varpi$.
  Fix any $\rho>\overline{\rho}$. By Remark \ref{remark-Alg2} (c),
  $G^0=\frac{\rho}{L_0+2\rho}\Gamma^0+\frac{L_0U^0\!-\!\nabla\!\widetilde{f}(U^0)}{L_0+2\rho}$.
  Then, for every $j\in\{1,\ldots,p\}$,
  \begin{equation}\label{Gj-ineq}
   \|G_{\!j}^0\|\ge\frac{\rho}{L_0\!+2\rho}\|\Gamma_{j}^0\|-\frac{1}{L_0\!+2\rho}\|[L_0U^0\!-\!\nabla\!\widetilde{f}(U^0)]_j\|\\
   \ge \frac{\rho\|\Gamma_{j}^0\|-\varpi}{L_0\!+2\rho}\ge\frac{\rho|P_{1\kappa}|}{L_0\!+2\rho},
  \end{equation}
  where the third inequality is using $\|\Gamma_{\!j}^0\|\ge 2|P_{1\kappa}|$
  and $\rho>\overline{\rho}$. This means that $G^0$ has no zero columns.
  Define $\overline{G}^0\!:=G^0D$ with $D={\rm Diag}(\frac{1}{\|G_1^0\|},\ldots,\frac{1}{\|G_p^0\|})$.
  Clearly, $V^{l,1}=V^1=\overline{G}^0=\frac{\rho}{L_0+2\rho}\Gamma^0D
  \!+\!\frac{1}{L_0+2\rho}[L_0U^0\!-\!\nabla\!\widetilde{f}(U^0)]D$.
  Let $\overline{\Gamma}^0\!:=\frac{1}{2}\Gamma^0D$. Clearly, ${\rm rank}(\overline{\Gamma}^0)
  ={\rm rank}(\Gamma^0)=1$. Then,
  \begin{align*}
   {\rm dist}(\overline{G}^0,\mathcal{R})
   &\le\|\overline{G}^0\!-\overline{\Gamma}^0\|_F
     \le\frac{2L_0\|V^0\|}{2(L_0\!+\!2\rho)}\|D\|_F
     +\frac{1}{L_0\!+\!2\rho}\big\|[L_0U^0\!-\!\nabla\!\widetilde{f}(U^0)]D\big\|_F\\
   & \le\frac{L_0\|V^0\|}{L_0\!+\!2\rho}\|D\|_F+\frac{\varpi}{L_0\!+\!2\rho}\|D\|
    \le\frac{L_0p}{\rho|P_{1\kappa}|}+\frac{\varpi}{\rho|P_{1\kappa}|}\le c_0,
  \end{align*}
  where the second inequality is using $\|\Gamma^0\|\le 2\|V^0\|$,
  the third and the fourth are using \eqref{Gj-ineq},
  and the last one is due to $\rho\ge\overline{\rho}$.
  The proof is then completed.  \qed
 \end{proof}
  \begin{proposition}\label{defined-Algfactor1}
   Fix an integer $l\ge 0$. Let $\{V^k\}$ be the sequence generated by Algorithm \ref{Alg2}
   from $V^0=V^{l}$, and let $\varepsilon\in\!(0,c_0]$ be a given tolerance,
   where $c_0$ is the constant from Lemma \ref{complexity}. Then,
   when $\rho_{l}\ge\widehat{\rho}:=\max\big\{\frac{\varpi}{(1-\sqrt{1\!-0.5p^{-1}\varepsilon})\sqrt{1\!-c_0}},
  \frac{2\varpi}{\varepsilon}\big\}$ with $\varpi=\!8(L_{\!\widetilde{f}}\!+\!L_0)p
  +2\sqrt{p}\|\nabla\!\widetilde{f}(\frac{1}{\sqrt{m}}E)\|_F$,
  \begin{itemize}
  \item [(i)] for each integer $k\ge 0$ with $\frac{\varepsilon}{2}\le p-\|V^k\|^2\le c_0$,
              $\|V^{k+1}\|^2\ge\|V^k\|^2+\big(1-\!\sqrt{1-\!0.5p^{-1}\varepsilon}\big)\sqrt{1-c_0}$;

    \item [(ii)] if there exists $P\!\in\mathbb{O}((V^{l})^{\mathbb{T}}V^{l})$ such that $P_1$ has no zero entries
               and $\rho_l>\max(\widehat{\rho},\overline{\rho})$ with $\overline{\rho}$ from Lemma \ref{complexity},
               there is an integer $1\le\overline{k}\le\lceil\frac{c_0}{(1-\sqrt{1-0.5p^{-1}\varepsilon})\sqrt{1-c_0}}\rceil+1$
               such that $\|V^k\|_F^2-\!\|V^k\|^2\le\varepsilon$ for all $k\ge\overline{k}$.
  \end{itemize}
 \end{proposition}
 \begin{proof}
 {\bf(i)} For each integer $k\ge 1$, from the definition of $V^{k+1}$,
  for any $V\in\mathcal{S}$,
  \begin{align*}
   \rho\langle \Gamma^{k},V^{k+1}-V\rangle
   &\le \langle\nabla\!\widetilde{f}(U^k),V\!-\!V^{k+1}\rangle
   +\frac{L_k}{2}\|V\!-\!U^{k}\|_F^2\!-\!\frac{L_k}{2}\|V^{k+1}\!-\!U^{k}\|_F^2\nonumber\\
   &\le\langle\nabla\!\widetilde{f}(U^k),V\!-\!V^{k+1}\rangle
   +\frac{L_k}{2}\big[\|V\|_F^2+2\|V^{k+1}\!-\!V\|_F\|U^k\|_F\big]\nonumber\\
   &\le \langle\nabla\!\widetilde{f}(U^k)\!-\!\nabla\!\widetilde{f}(E/{\sqrt{m}})
   \!+\!\nabla\!\widetilde{f}(E/{\sqrt{m}}),V\!-\!V^{k+1}\rangle +6.5L_0p\nonumber\\
   &\le 2\sqrt{p}(4L_{\widetilde{f}}\sqrt{p}\!+\!\|\nabla\!\widetilde{f}(E/{\sqrt{m}})\|_F)+6.5L_0p\le\varpi,
  \end{align*}
  where the third inequality is by $L_k\le L_0$ for all $k$
  and $\|V\|_F^2=p$ for $V\in\mathcal{S}$. Then
  \begin{equation}\label{temp-WV}
   -\!\langle \Gamma^{k},V\rangle
   \le\rho^{-1}\varpi+\|\Gamma^k\|_*\|V^{k+1}\|
   =\rho^{-1}\varpi+2\|V^k\|\|V^{k+1}\|
  \end{equation}
  where the equality is by the choice of $\Gamma^k$ in Remark \ref{remark-Alg2} (ii).
  Let $V^k$ have the SVD given by
  $Q[{\rm Diag}((\sigma_1(V^k),\ldots,\sigma_m(V^k))^{\mathbb{T}})\ \ 0]U^{\mathbb{T}}$
  with $\sigma_1(V^k)\ge\cdots\ge\sigma_m(V^k)$. Write
  $Q=[q_1\,\cdots\,q_m]\in\mathbb{O}^m$ and $U=[u_1\,\cdots\,u_p]\in\mathbb{O}^p$.
  Then, for every $j\in\{1,\ldots,p\}$,
  \begin{equation}\label{sigmaVu}
   [\sigma_1(V^k)]^2u_{1j}^2=1-{\textstyle\sum_{i=2}^m}[\sigma_i(V^k)]^2u_{ij}^2\ge 1-{\textstyle\sum_{i=2}^m}[\sigma_i(V^k)]^2\ge 1-c_0>0
  \end{equation}
  where the next to last inequality is using $\|V^k\|_F^2-\|V^k\|^2\le c_0$.
  Take $\widehat{V}=q_1\widehat{u}_1^\mathbb{T}$ with
  $\widehat{u}_{1j}=\frac{u_{1j}}{|u_{1j}|}$ for each $j$. Clearly, $\widehat{V}\in\mathcal{S}$.
  From \eqref{temp-WV} with $V=\widehat{V}$ and $\Gamma^k=-2V^ku_1u_1^{\mathbb{T}}$,
  \begin{align}\label{temp-snromV}
   \rho^{-1}\varpi+2\|V^k\|\|V^{k+1}\|&\ge-\langle \Gamma^k,\widehat{V}\rangle
   =2\|V^k\|(u_1^{\mathbb{T}}\widehat{u}_1)
   =2\|V^k\|^2{\textstyle\sum_{j=1}^p}\frac{|u_{1j}|}{\|V^k\|}\nonumber\\
   &=2\|V^k\|^2+2\|V^k\|^2\sum_{j=1}^p\Big(\frac{u_{1j}^2}{\|V^k\||u_{1j}|}-u_{1j}^2\Big)
  \end{align}
  where the third equality is by $\sum_{j=1}^pu_{1j}^2=1$. Since $\|u_1\|=1$,
  there is an index $\widehat{j}$ such that $u_{1\widehat{j}}^2\le\frac{1}{p}$.
  Note that $\|V^k\||u_{1j}|\le 1$ for each $j$ by the first equality of \eqref{sigmaVu}.
  So,
 \begin{align*}
  \rho^{-1}\varpi+2\|V^k\|\|V^{k+1}\|&
   \ge 2\|V^k\|^2+2\|V^k\|^2\big(\frac{1}{\|V^k\||u_{1\widehat{j}}|}-1\big)u_{1\widehat{j}}^2\\
   &\ge 2\|V^k\|^2+2\big(1-{\|V^k\|}/{\sqrt{p}}\big)\sqrt{1-c_0}\\
   &\ge 2\|V^k\|^2+2(1-\!\sqrt{1-0.5p^{-1}\varepsilon})\sqrt{1-c_0},
  \end{align*}
  where the second inequality is by \eqref{sigmaVu},
  and the last is since $p-\|V^k\|^2\ge\frac{\varepsilon}{2}$.
  Along with $\rho\ge\frac{\varpi}{(1-\sqrt{1\!-0.5p^{-1}\varepsilon})\sqrt{1\!-c_0}}$,
  we get $\|V^k\|\|V^{k+1}\|\ge \|V^k\|^2+\frac{1}{2}(1-\!\sqrt{1\!-0.5p^{-1}\varepsilon})\sqrt{1-c_0}$.
  Together with $\|V^k\|\|V^{k+1}\|\le \frac{1}{2}\|V^k\|^2+\frac{1}{2}\|V^{k+1}\|^2$,
  the desired result follows.

  \noindent
  {\bf(ii)} Let $\eta:=\big(1\!-\!\sqrt{1-0.5p^{-1}\varepsilon}\big)\sqrt{1-c_0}$
  and $\widehat{k}:=\lceil\frac{c_0}{(1-\sqrt{1-0.5p^{-1}\varepsilon})\sqrt{1-c_0}}\rceil\!+\!1$.
  We first argue that there exists $1\le\overline{k}\le\widehat{k}$ such that
  $p-\|V^{\overline{k}}\|^2\le\varepsilon$. If not, for all $1\le k\le\widehat{k}$,
  we have $p-\|V^k\|^2>\varepsilon$. By Lemma \ref{complexity},
  $p-\|V^1\|^2\le c_0$. Thus, from part (i), it follows that
  $p-\|V^k\|^2\le c_0$ for all $1\le k\le\widehat{k}$. Using part (i) again,
  $\|V^{k+1}\|^2\ge\|V^k\|^2+\eta$ for all $1\le k\le\widehat{k}$. From this,
  $\|V^{\widehat{k}}\|^2\ge \|V^1\|^2\!+(\widehat{k}\!-\!1)\eta\ge p-c_0+(\widehat{k}\!-\!1)\eta$,
  which is impossible due to $\|V^{\widehat{k}}\|^2<p-\varepsilon$, so the stated $\overline{k}$ exists.
  Next we argue by induction that $p-\|V^k\|^2\le\varepsilon$ for all $k\ge\overline{k}$.
  Suppose that $p-\|V^j\|^2\le\varepsilon$ for $j\ge\overline{k}$.
  We show that $p-\|V^{j+1}\|^2\le\varepsilon$ by two cases.
  If $p-\|V^j\|^2<\varepsilon/2$, by invoking \eqref{temp-WV} with $V=V^j$, we have
  $2\|V^j\|\|V^{j+1}\|\ge2\|V^j\|^2-\rho^{-1}\varpi$,
  which implies that $\|V^{j+1}\|^2\ge \|V^j\|^2-\rho^{-1}\varpi$,
  so $p-\|V^{j+1}\|^2\le p-\|V^j\|^2+\rho^{-1}\varpi\le\varepsilon/2+\rho^{-1}\varpi\le\varepsilon$.
  If $p-\|V^j\|^2\ge\varepsilon/2$, since $p-\|V^j\|^2\le\varepsilon<c_0$,
  from part (i) we have $p-\|V^{j+1}\|^2\le p-\|V^j\|^2-\eta\le\varepsilon$.
  The proof is completed. \qed
  \end{proof}

 The following theorem states that the rank-one projection of the normal output
 of Algorithm \ref{Alg-factor} is also an approximately feasible solution
 of the problem \eqref{UBPP}, and provides a quantitative bound estimation
 for its objective value to the optimal value of the problem \eqref{UBPP}.
 \begin{theorem}\label{obj-bound2}
  Let $\upsilon^*$ be the optimal value of \eqref{UBPP} and let $V^{l_{\!f}}$
  be a normal output of Algorithm \ref{Alg-factor}. For each $l\ge 0$, let $\{(V^{l,k},\Gamma^{l,k})\}$
  be generated by Algorithm \ref{Alg2} with $V^{l,0}=V^{l}$
  and $\beta_k\equiv 0$. If there exists $l^*\in\{0,1,\ldots,l_{\!f}\}$ such that
  $f((V^{l^*})^{\mathbb{T}}V^{l^*})\le \upsilon^*$, then with $r^*\!={\rm rank}(V^{l^*})$ it holds that
  \begin{subequations}
   \begin{align*}
   \!f(x^{l_{\!f}}(x^{l_{\!f}})^{\mathbb{T}})-\upsilon^*
    \le \rho_{l_{\!f}}\|V^{l_{\!f}}\|^2-\rho_{l^*}p/r^*+\!{\textstyle\sum_{j=l^*}^{l_{\!f}-1}}(\rho_{j}\!-\!\rho_{j+1})\|V^{j+1}\|^2
        +\alpha_{\!f}\epsilon;\\
   \!\|x^{l_{\!f}}\circ x^{l_{\!f}}-e\|\le\epsilon\ \ {\rm with}\ \ x^{l_{\!f}}=\|V^{l_{\!f}}\|P_1
   \ \ {\rm for}\ P\in\!\mathbb{O}((V^{l_{\!f}})^{\mathbb{T}}V^{l_{\!f}}).\qquad
   \end{align*}
  \end{subequations}
 \end{theorem}
  \begin{proof}
  Fix any $l\in\{0,1,\ldots,l_{\!f}\}$. For each $k\ge 0$,
  from $\beta_k\equiv 0$ and \eqref{key-ineq0},
  \begin{equation*}
   \langle\nabla\!\widetilde{f}(V^{l,k})+\rho_{l}\Gamma^{l,k},V^{{l,k}+1}\!-\!V^{l,k}\rangle
   +\rho_{l}(\|V^{{l,k}+1}\|_F^2-\|V^{l,k}\|_F^2)+\frac{L_{{l,k}}}{2}\|V^{{l,k}+1}\!-\!V^{l,k}\|_F^2\le 0.
 \end{equation*}
  Notice that $\widetilde{f}(V^{l,k+1})\le \widetilde{f}(V^{l,k})
  +\langle\nabla\!\widetilde{f}(V^{l,k}),V^{l,k+1}\!-\!V^{l,k}\rangle
  +\frac{L_{\widetilde{f}}}{2}\|V^{l,k+1}\!-\!V^{l,k}\|_F^2$ by using \eqref{wfrho}
  with $V=V^{l,k+1}$ and $Z=V^{l,k}$. From the last inequality and $L_{l,k}\ge L_{\widetilde{f}}$,
  \[
    \widetilde{f}(V^{l,k+1})-\widetilde{f}(V^{l,k})+\rho_l\langle\Gamma^{l,k},V^{{l,k}+1}\!-\!V^{l,k}\rangle \le 0,
  \]
  where $\|V^{{l,k}+1}\|_F^2\!=\!\|V^{{l,k}}\|_F^2\!=\!p$ is also used.
  Notice that $\langle\Gamma^{l,k},V^{l,k}\rangle\!=\!-2\|V^{l,k}\|^2$ and
  \(
  \langle\Gamma^{l,k},V^{{l,k}+1}\rangle\\
  \ge -2\|V^{l,k}\|\|V^{{l,k}+1}\|\ge-\|V^{l,k}\|^2-\|V^{{l,k}+1}\|^2.
  \)
   Then, it holds that
  \(
    \widetilde{f}(V^{l,k+1})\!-\!\rho_l\|V^{{l,k}+1}\|^2
   \le \widetilde{f}(V^{l,k})-\rho_l\|V^{l,k}\|^2.
  \)
  By using this recursion formula,
  \[
    \widetilde{f}(V^{l,k+1})\!-\!\rho_l\|V^{{l,k}+1}\|^2
    \le\cdots\le \widetilde{f}(V^{l,0})\!-\!\rho_l\|V^{0}\|^2=\widetilde{f}(V^{l})-\rho_{l}\|V^l\|^2.
  \]
  By Theorem \ref{theorem-Alg2}, the sequence $\{V^{l,k}\}$ is convergent as $k\to\infty$.
  Let $V^{l,*}$ denote its limit. Then $V^{l+1}=V^{l,*}$.
  From the last inequality, for each $l\in\{0,1,\ldots,l_{\!f}\}$,
  \begin{equation}\label{lineq1}
    \widetilde{f}(V^{l+1})-\rho_l\|V^{l+1}\|^2\le \widetilde{f}(V^{l})-\rho_{l}\|V^l\|^2.
  \end{equation}
  Notice that $X^{l_{\!f}}=\sum_{i=1}^p\lambda_i(X^{l_{\!f}})P_iP_i^{\mathbb{T}}\in\Omega$.
  From the Lipschitz continuity of $f$ relative to $\Omega$ with modulus $\alpha_{\!f}$,
  it follows that
  \begin{align}\label{xlf-ineq}
  f(X^{l_{\!f}})&=f\big({\textstyle\sum_{i=1}^p}\lambda_i(X^{l_{\!f}})P_iP_i^{\mathbb{T}}\big)
  =f\big(\lambda_1(X^{l_{\!f}})P_1P_1^{\mathbb{T}}+{\textstyle\sum_{i=2}^p}\lambda_i(X^{l_{\!f}})P_iP_i^{\mathbb{T}}\big)
    \nonumber\\
  &\ge f(x^{l_{\!f}}(x^{l_{\!f}})^{\mathbb{T}})-
      \alpha_{\!f}\|{\textstyle\sum_{i=2}^p}\lambda_i(X^{l_{\!f}})P_iP_i^{\mathbb{T}}\|_F
   \ge f(x^{l_{\!f}}(x^{l_{\!f}})^{\mathbb{T}})-\alpha_{\!f}\epsilon.
  \end{align}
  On the other hand, adding $(\rho_{l}-\rho_{l+1})\|V^{l+1}\|^2$ to
  the both sides of \eqref{lineq1} yields that
  \[
    \widetilde{f}(V^{l+1})-\rho_{l+1}\|V^{l+1}\|^2
    \le \widetilde{f}(V^{l})-\rho_{l}\|V^{l}\|^2+(\rho_{l}-\rho_{l+1})\|V^{l+1}\|^2
    \quad\forall l\in\{0,\ldots,l_{\!f}\}.
  \]
  Then, we obtain
  \(
    \widetilde{f}(V^{l_f})-\rho_{l_{\!f}}\|V^{l_f}\|^2
    \le \widetilde{f}(V^{l^*})-\rho_{l^*}\|V^{l^*}\|^2
    +{\textstyle\sum_{j=l^*}^{l_{\!f}-1}}(\rho_{j}-\rho_{j+1})\|V^{j+1}\|^2.
  \)
  Recall that $\widetilde{f}(V^{l_f})=f(X^{l_f})$. Together with \eqref{xlf-ineq} and $\|V^{l^*}\|\ge\sqrt{p/r^*}$,
  the first inequality follows. Since ${\rm diag}(X^{l_{\!f}})=e$, we have
  $\|x^{l_{\!f}}\circ x^{l_{\!f}}-e\|=\|\sum_{i=2}^p\lambda_i(X^{l_{\!f}})P_i\circ P_i\|
  \le\epsilon$.  \qed
  \end{proof}

  If $f$ is convex, by computing $X^0\in{\displaystyle\mathop{\arg\min}_{X\in\Omega}}f(X)$ and letting
  $\Lambda^0\!=\!{\rm Diag}(\lambda_1(X^0),\ldots,\lambda_m(X^0))$, then $l^*=0$ and
  $V^{l^*}\!=\!\sqrt{\Lambda^0}P_I^{\mathbb{T}}$ with $P\!\in\!\mathbb{O}(X^0)$
  satisfies the assumption of Theorem \ref{obj-bound2}. Hence, for a convex $f$,
  the normal output of Algorithm \ref{Alg-factor} yields a desirable approximate feasible solution.
 \section{Numerical experiments}\label{sec5}

  This section tests the performance of Algorithm \ref{Alg-factor} armed with
  Algorithm \ref{Alg2} (\textbf{dcFAC} for short). To confirm its efficiency,
  we compare its performance with that of SDP relaxation method armed with
  random rounding technique in \cite{Goemans95} (\textbf{SDPRR} for short);
  see Section \ref{sec5.2} for its description. We also compare the performance
  of dcFAC with that of Algorithm \ref{Alg} below, a DC relaxation approach based on
  model \eqref{epenalty}, for which every penalty subproblem \eqref{epenalty-subprob}
  is solved by Algorithm \ref{Alg-X}, an MM method with extrapolation. When $f$ is nonlinear
  (say, the instances in Section \ref{sec5.6}), we use Algorithm \ref{Alg-X} to solve
  the subproblem \eqref{epenalty-subprob}, where $L_{\!f}$ is the Lipschitz constant
  of $\nabla\!f$ in $\mathbb{B}_{\Omega}$. For the convergence analysis of Algorithm \ref{Alg-X},
  see Appendix B. Considering that QSDPNAL \cite{LiSunToh18} is not well adapted to quadratic
  SDP subproblems of Algorithm \ref{Alg-X}, we use the dual semismooth Newton method in \cite{QiSun06}
  to solve them. In the sequel, Algorithm \ref{Alg} armed with  Algorithm \ref{Alg-X}
  is abbreviated to \textbf{dcSNCG}. Our code can be downloaded from \url{https://github.com/SCUT-OptGroup/rankone_UPPs}.
  When $f$ is a linear function (say, the instances in Section \ref{sec5.2}-\ref{sec5.5}),
  the subproblem \eqref{epenalty-subprob} is solved by Algorithm \ref{Alg-X}
  with $L_k\equiv L_{\!f}=0$ and $\beta_k\equiv 0$. Although there is no convergence certificate
  for such Algorithm \ref{Alg-X}, we adopt it just for numerical comparisons.
  In this case, the linear SDP subproblems in Algorithm \ref{Alg-X} are solved with
  the software SDPT3 \cite{Toh99}, and Algorithm \ref{Alg} equipped with
  such Algorithm \ref{Alg-X} and SDPT3 is abbreviated to \textbf{dcSDPT3}.
 \begin{algorithm}[H]
 \renewcommand{\thealgorithm}{2}
  \caption{(DC relaxation approach based on \eqref{epenalty})}
  \label{Alg}
  \begin{algorithmic}
  \normalsize
  \STATE{Choose $\epsilon\in(0,1),l_{\rm max}\in\mathbb{N},\rho_{\rm max}>0,\sigma>1,\rho_0>0$
  and $X^0\in\Omega$.}
  \FOR{$l=0,1,2,\ldots,l_{\rm max}$}
  \STATE{Starting from $X^l$, seek a stationary point $X^{l+1}$ of the nonconvex problem
             \begin{equation}\label{epenalty-subprob}
             \min_{X\in\Omega}\Big\{f(X)+\rho_{l}(\langle I,X\rangle-\|X\|)\Big\}.
            \end{equation}}
  \STATE{If $\langle I,X^{l+1}\rangle-\|X^{l+1}\|\le\epsilon$, then stop. Otherwise,
         let $\rho_{l+1}\leftarrow\min\{\sigma\rho_l,\rho_{\rm max}\}$. }
  \ENDFOR
  \end{algorithmic}
  \end{algorithm}
  \vspace{-0.3cm}
 \begin{algorithm}
 \renewcommand{\thealgorithm}{B}
 \caption{(An MM method with extrapolation for \eqref{epenalty-subprob})}
 \label{Alg-X}
 \begin{algorithmic}
 \normalsize
 \STATE{Fix $l\ge 0$. Choose $0\le\!\beta_0\!\le\!\overline{\beta}<1$ and $L_0\ge L_{\!f}$.
  Set $\rho=\rho_l$ and $X^{-1}=X^0=X^l$.}
 \FOR{$k=0,1,2,\ldots$}
 \STATE{Choose an element $W^{k}\in\partial\psi(X^k)$}.
 \STATE{Let $Y^k=X^k+\beta_k(X^k\!-\!X^{k-1})$. Compute an optimal solution of the convex SDP:
             \begin{equation}\label{Xk-subprob}
             \!X^{k+1}=\mathop{\arg\min}_{X\in\Omega}\Big\{\langle\nabla\!f(Y^k)\!+\!\rho(I\!+\!W^k),X\rangle
              +({L_k}/{2})\|X\!-\!Y^k\|_F^2\Big\}.
              \end{equation}}
 \STATE{Update $\beta_{k}$ by $\beta_{k+1}\in[0,\overline{\beta}]$ and $L_k$ by $L_{k+1}\in[L_{\!f},L_0]$}.
 \ENDFOR
 \end{algorithmic}
 \end{algorithm}

  All tests are performed in MATLAB on a workstation running on 64-bit Windows
  Operating System with an Intel Xeon(R) W-2245 CPU 3.90GHz and 128 GB RAM.
  We measure the performance of a solver by the relative gap and infeasibility
  of its outputs and the CPU time (in seconds) taken. Let $x^*\!=\!\|V^*\|Q_1^*$
  or $\sqrt{\|X^*\|}P_1^*$ with $Q^*\!\in\!\mathbb{O}((V^*)^{\mathbb{T}}V^*)$ and
  $P^*\!\in\!\mathbb{O}(X^*)$, where $V^*$ is the output of dcFAC and $X^*$ is
  an output for one of other three solvers. The relative gap and infeasibility
  of $x^*$ are defined by $\textbf{gap}\!:=\frac{\rm Obj-Bval}{\rm Bval}$ and
  $\textbf{infeas}\!:=\big\|(x^*\circ x^*)^{1/2}-e\big\|_{\infty}$,
  where ${\rm Bval}$ means the known best value of \eqref{UBPP},
  and Obj denotes the objective value of \eqref{UBPP} at $x^*$. For the subsequent tests,
  we use the default setting for the softwares SDPT3 and SDPNAL+.
 \subsection{Implementation of dcFAC and dcSNCG}\label{sec5.1}
  We first focus on the choice of parameters in Algorithm \ref{Alg-factor}
  and \ref{Alg}. Preliminary tests indicate that smaller $\rho_0$ and $\sigma$
  often lead to better relative gaps for Algorithm \ref{Alg-factor} and \ref{Alg}.
  Since it is time consuming to search the best $\rho$, an appropriately small $\rho_0$
  becomes a reasonable choice. We choose $\rho_0\!=\!0.001,\sigma\!=\!1.005$
  for Algorithm \ref{Alg-factor}, but $\rho_0=0.1,\sigma=1.05$
  for Algorithm \ref{Alg} armed with SNCG since it requires much more time
  for those examples with $n\ge 2000$. For Algorithm \ref{Alg} armed with SDPT3,
  we use $\rho_0\!=\!0.001,\sigma\!=\!1.005$ for solving the examples with $n\!<\!500$,
  but $\rho_0\!=\!0.1,\sigma\!=\!1.05$ for solving the examples with $n\ge 500$.
  We set $\epsilon=10^{-8},\rho_{\rm max}=10^6$ and $l_{\rm max}=10^4$
  for Algorithm \ref{Alg-factor} and \ref{Alg}. In addition, we take
  $m=\!\max(\min(50,{\rm round}(p/2)),2)$ by considering that
  a smaller $m$ makes \eqref{BPP-factor} vulnerable to much
  worse critical points, but a larger $m$ requires more computation cost.
  The starting point $V^0$ of Algorithm \ref{Alg-factor} is chosen
  to be $\widetilde{V}{\rm Diag}(\|\widetilde{V}_{1}\|^{-1},\ldots,\|\widetilde{V}_{p}\|^{-1})$
  where $\widetilde{V}\!\in\mathbb{R}^{m\times p}$ is generated in MATLAB
  command $\textrm{randn}(m,p)$ with a fixed seed for all test problems;
  and the starting point $X^0$ of Algorithm \ref{Alg} is chosen to be $(V^0)^{\mathbb{T}}V^0$.

  The parameter $\beta_k$ in Algorithm \ref{Alg2} and \ref{Alg-X} is given
  by Nesterov's accelerated strategy \cite{Nesterov83}.
  Although their convergence analysis requires a restriction on $\beta_k$,
  numerical tests indicate that they still converge without it. Hence,
  we do not impose any restriction on such $\beta_k$ during their implementation,
  and leave this gap for a future topic. For the parameter $L_k$ of Algorithm \ref{Alg2},
  when $L_{\!\widetilde{f}}$ is known (say, the instances in Section \ref{sec5.2}-\ref{sec5.5}),
  we set it to be a fixed constant, otherwise search a desired $L_k$ by
  the descent lemma. Specifically, we set $L_k\equiv 2.001\|C\|$ for the instances
  in Section \ref{sec5.2}-\ref{sec5.5} since $L_{\!\widetilde{f}}=2\|C\|$,
  and search a desired $L_k$ with $L_0\equiv 6p(\|C_1\|\|C_2\|_F\!+\!\|C_1\|_F\|C_2\|)$
  for the instances in Section \ref{sec5.6}. For the parameter $L_k$ of Algorithm \ref{Alg-X},
  we take $\|C_1\|\|C_2\|_F\!+\!\|C_1\|_F\|C_2\|$ for the problems in Section \ref{sec5.6}
  since it is exactly the Lipschitz constant $L_{\!f}$.

  During the implementation of Algorithm \ref{Alg2}, we seek an approximate
  stationary point of subproblem \eqref{epenalty-factor}. According to
  the optimality conditions of the subproblem \eqref{epenalty-factor}, we terminate Algorithm \ref{Alg2}
  whenever $k\le k_{\rm max}$ or the following condition is satisfied
  \[
   \|\nabla\widetilde{f}(V^{k+1})-\nabla\widetilde{f}(U^k)-L_k(V^{k+1}-U^k)+\rho(\Gamma^{k+1}-\Gamma^k)\|_F
    \le\tau_k\max(1,\eta),
 \]
 where $\tau_{k+1}=\max(10^{-5},0.995\tau_k)$ with $\tau_0=0.005$,
 and $\eta>0$ is a constant related to test instances. Among others,
 $\eta=\|C\|_F$ for the examples in Section \ref{sec5.2}-\ref{sec5.4},
 and $\eta=\max_{1\le i\le q}\|C_i\|_F$ for those in Section \ref{sec5.5}.
 A similar stopping condition, except $\tau_{k+1}=\max(10^{-5},0.9\tau_k)$,
 is also used for Algorithm \ref{Alg-X} and SDPT3 to solve \eqref{epenalty-subprob}.
 Consider that those penalty problems with smaller $\rho$ are actually used to
 seek an appropriate $\rho$. When $\|V^{l}\|_F^2-\|V^{l}\|^2$ has a larger value
 (corresponding to a smaller $\rho$), we can calculate a very rough approximate
 stationary point of \eqref{epenalty-factor}. Inspired by this,
 during the testing, we take $k_{\rm max}=3$ when $\|V^{l}\|_F^2-\|V^{l}\|^2>1$
 for Algorithm \ref{Alg2} and \ref{Alg-X}, but respectively set $k_{\rm max}=3000$
 and $k_{\rm max}=1000$ for them when $\|V^{l}\|_F^2-\|V^{l}\|^2\le 1$.
 \subsection{Comparisons with SDPRR and dcSDPT3 for Biq instances}\label{sec5.2}

  In this part, we compare the performance of dcFAC with that of SDPRR
  and dcSDPT3 for the problem $\max_{z\in\{0,1\}^n}z^{\mathbb{T}}Az$,
  which can be reformulated as \eqref{DC-BPP} with $p=n+1$ and $f(X)=\langle C,X\rangle$
  for $C=-\frac{1}{4}\left(\begin{matrix}
          0& e^{\mathbb{T}}A\\
          Ae & A
       \end{matrix}\right)$.
  The matrix $A$ is from the Biq Mac Library.
  The SDPRR first uses the software SDPNAL+ \cite{Yang15,Sun20} to solve
  the SDP yielded by removing the DC constraint in \eqref{rank-BQP1} but
  adding the valid inequalities $X_{1j}+X_{1k}+X_{jk}\ge-1,X_{1j}-X_{1k}-X_{jk}\ge-1,
  -X_{1j}+X_{1k}-X_{jk}\ge-1$ and $-X_{1j}-X_{1k}+X_{jk}\ge-1$ for all $1<j<k$,
  and then impose \textbf{50} times random rounding technique \cite{Goemans95}
  on the solution and select the best one from $50$ feasible solutions.
 Table \ref{table1} reports the objective values of the outputs of
 three solvers and the CPU time taken by them, and the optimal values of these instances,
 where the gap value in red means the best for an instance.

 We see that dcFAC and dcSDPT3 have much better performance than SDPRR does
 in terms of the quality of the outputs, and among the 119 instances, the outputs of
 dcFAC and dcSDPT3 respectively have \textbf{103} and \textbf{106} best ones,
 and their relative gaps to the optimal values are at most \textbf{1.824\%}
 except \textbf{gka9b} and \textbf{gka10b}. Since the data matrix from \textbf{gka1b}-\textbf{gka10b}
 has a special structure, i.e., the diagonal entries are from $[-63,0]$ while
 the off-diagonal entries are from $[0,100]$, the outputs of dcSDPT3 and dcFAC
 have a zero objective value for them. The CPU time of dcFAC is far less
 than that of dcSDPT3 and SDPRR, and for those examples with $n=250$,
 dcFAC requires at most $6.0s$ but dcSDPT3 and SDPRR require at least $900s$.
 \setlength{\tabcolsep}{1mm}{
 \begin{center}
  \setlength{\abovecaptionskip}{0pt}%
 \setlength{\belowcaptionskip}{2pt}%
 \captionsetup{font={scriptsize}}
 \renewcommand\arraystretch{1.5}
 \tiny
 \begin{longtable}{|c|c|c|cc|cc|cc||c|c|c|cc|cc|cc|}
  \caption{Numerical results of dcSDPT3, dcFAC and SDPRR for Biq Mac Library instances}\label{table1}\\
  \hline
  & & &\multicolumn{2}{c|}{\bf dcSDPT3}&\multicolumn{2}{c|}{\bf dcFAC}&\multicolumn{2}{c||}{\bf SDPRR}
  & & & &\multicolumn{2}{c|}{\bf dcSDPT3}&\multicolumn{2}{c|}{\bf dcFAC}&\multicolumn{2}{c|}{\bf SDPRR}\\
  \hline
 {\bf Name}&$n$&{\bf Optval}& {\bf Obj}&time& {\bf Obj}&time& {\bf Obj}&time &{\bf Name}&$n$&{\bf Optval}
 & {\bf Obj}&time& {\bf Obj}&time& {\bf Obj}&time \\
 \hline
 \textrm{bqp100-1}&100& 7970& 7848& 333.3& 7848& 4.4& {\bf\color{red}7938}& 408.6
 &\textrm{bqp100-2}&100& 11036& 11032& 330.1& 11032& 4.2& {\bf\color{red}11036}& 214.0\\
 \hline

 \textrm{bqp100-3}&100& 12723& {\bf\color{red}12723}& 316.9& {\bf\color{red}12723}& 4.2& {\bf\color{red}12723}& 2.3
&\textrm{bqp100-4}&100& 10368& {\bf\color{red}10368}& 306.1& {\bf\color{red}10368}& 4.3& {\bf\color{red}10368}& 6.2\\
 \hline

 \textrm{bqp100-5}&100& 9083& 9045& 324.6& 9045& 4.5& {\bf\color{red}9083}& 44.5
&\textrm{bqp100-6}&100& 10210& {\bf\color{red}10202}& 327.5& {\bf\color{red}10202}& 4.0& 10164& 364.6\\
 \hline

 \textrm{bqp100-7}&100& 10125& 10060& 338.2& 10060& 4.3& {\bf\color{red}10121}& 462.2
&\textrm{bqp100-8}&100& 11435& 11415& 318.7& 11415& 4.1& {\bf\color{red}11435}& 8.9\\
 \hline

 \textrm{bqp100-9}&100& 11455& {\bf\color{red}11455}& 317.6& {\bf\color{red}11455}& 4.6& {\bf\color{red}11455}& 3.3
&\textrm{bqp100-10}&100& 12565& 12521& 327.8& 12521& 4.3& {\bf\color{red}12565}& 11.6\\
 \hline

 \textrm{bqp250-1}&250& 45607& {\bf\color{red}45547}& 1066.1&{\bf\color{red}45547} &4.7 &44588 &415.4
&\textrm{bqp250-2}&250& 44810& {\bf\color{red}44810}& 1080.3&44774 &4.9 &44136 &2085.6 \\
 \hline

 \textrm{bqp250-3}&250& 49037& {\bf\color{red}48977}& 1046.6&{\bf\color{red}48977} &5.0 &48935 &1984.0
&\textrm{bqp250-4}&250& 41274& {\bf\color{red}41270}& 1029.6&{\bf\color{red}41270} &4.7 &40589 &2221.2 \\
 \hline

 \textrm{bqp250-5}&250& 47961& {\bf\color{red}47815}& 1071.5&{\bf\color{red}47815} &5.5 &47645 &2246.1
&\textrm{bqp250-6}&250& 41014& {\bf\color{red}40906}& 1027.6&{\bf\color{red}40906} &4.7 &40236 &2127.2 \\
 \hline

 \textrm{bqp250-7}&250& 46757& {\bf\color{red}46687}& 1068.9&{\bf\color{red}46687} &4.7 &46505 &2007.4
&\textrm{bqp250-8}&250& 35726& {\bf\color{red}35556}& 1031.3&{\bf\color{red}35556} &5.2 &33982 &1891.7 \\
 \hline

 \textrm{bqp250-9}&250& 48916& {\bf\color{red}48916}& 1066.6&{\bf\color{red}48916} &4.5 &48002 &1733.1
&\textrm{bqp250-10}&250& 40442& {\bf\color{red}40336}& 1064.2&{\bf\color{red}40336} &4.8 &39722 &2145.8 \\
 \hline

 \textrm{be100.1}&100& 19412& {\bf\color{red}19412}& 319.4& {\bf\color{red}19412}& 3.9& {\bf\color{red}19412}& 479.5
&\textrm{be100.2}&100& 17290& {\bf\color{red}17290}& 329.2& {\bf\color{red}17290}& 4.0& 17256& 63.2\\
 \hline

 \textrm{be100.3}&100& 17565& {\bf\color{red}17532}& 335.4& {\bf\color{red}17532}& 3.8& 17469& 430.9
&\textrm{be100.4}&100& 19125& {\bf\color{red}19122}& 336.9& {\bf\color{red}19122}& 3.4& 19062& 381.0\\
 \hline

 \textrm{be100.5}&100& 15868& {\bf\color{red}15812}& 331.7& {\bf\color{red}15812}& 3.5& 15786& 345.9
&\textrm{be100.6}&100& 17368& {\bf\color{red}17368}& 333.4& {\bf\color{red}17368}& 3.8& 17316& 440.2\\
 \hline

 \textrm{be100.7}&100& 18629& {\bf\color{red}18601}& 349.5& {\bf\color{red}18601}& 3.6& 18463& 369.9
&\textrm{be100.8}&100& 18649& {\bf\color{red}18641}& 341.0&{\bf\color{red} 18641}& 3.6& 18225& 347.6\\
 \hline

 \textrm{be100.9}&100& 13294& {\bf\color{red}13254}& 357.7& {\bf\color{red}13254}& 4.0& 13110& 314.2
&\textrm{be100.10}&100& 15352& {\bf\color{red}15352}& 333.1& {\bf\color{red}15352}& 3.8& 15132& 358.3\\
  \hline

 \textrm{be120.3.1}&120& 13067& {\bf\color{red}13067}& 385.5&{\bf\color{red}13067} &3.8 &12995 &517.7
&\textrm{be120.3.2}&120& 13046& {\bf\color{red}13046}& 383.0&{\bf\color{red}13046} &4.0 &{\bf\color{red}13046} &414.1 \\
 \hline

 \textrm{be120.3.3}&120& 12418& {\bf\color{red}12418}& 385.8&{\bf\color{red}12418} &4.7 &12372 &148.0
&\textrm{be120.3.4}&120& 13867& {\bf\color{red}13867}& 367.4&{\bf\color{red}13867} &3.7 &13771 &511.2 \\
 \hline

 \textrm{be120.3.5}&120& 11403& {\bf\color{red}11384}& 393.5&{\bf\color{red}11384} &3.6 &11336 &541.0
&\textrm{be120.3.6}&120& 12915& {\bf\color{red}12866}& 389.5&{\bf\color{red}12866} &3.6 &12811 &164.4 \\
 \hline

 \textrm{be120.3.7}&120& 14068& {\bf\color{red}14054}& 392.5&{\bf\color{red}14054} &3.7 &{\bf\color{red}14054} &585.3
&\textrm{be120.3.8}&120& 14701& 14560& 384.4&14560& 3.8 &{\bf\color{red}14635} &491.8 \\
 \hline

 \textrm{be120.3.9}&120& 10458& {\bf\color{red}10375}& 402.9&{\bf\color{red}10375} &3.7 &10284 &472.1
&\textrm{be120.3.10}&120& 12201& {\bf\color{red}12201}& 403.1&{\bf\color{red}12201} &4.2 &12154 &516.1 \\
 \hline

 \textrm{be120.8.1}&120& 18691& {\bf\color{red}18658}& 424.1&{\bf\color{red}18658} &4.2 &18413 &176.6
&\textrm{be120.8.2}&120& 18827& {\bf\color{red}18797}& 399.7&{\bf\color{red}18797} &3.8 &18589 &453.2 \\
 \hline

 \textrm{be120.8.3}&120& 19302& {\bf\color{red}19228}& 420.5&{\bf\color{red}19228} &3.7 &19179 &490.2
&\textrm{be120.8.4}&120& 20765& {\bf\color{red}20765}& 396.1&{\bf\color{red}20765} &3.3 &20610 &515.8 \\
 \hline

 \textrm{be120.8.5}&120& 20417& {\bf\color{red}20381}& 418.9&{\bf\color{red}20381} &3.7 &20285 &575.0
&\textrm{be120.8.6}&120& 18482& {\bf\color{red}18482}& 403.2&{\bf\color{red}18482} &4.0 &18337 &187.8 \\
 \hline

 \textrm{be120.8.7}&120& 22194& {\bf\color{red}22131}& 404.6&{\bf\color{red}22131} &3.5 &22000 &456.1
&\textrm{be120.8.8}&120& 19534& {\bf\color{red}19236}& 456.0&{\bf\color{red}19236} &4.3 &19107 &467.9 \\
 \hline

 \textrm{be120.8.9}&120& 18195& {\bf\color{red}18181}& 402.0&{\bf\color{red}18181} &3.6 &17938 &491.7
&\textrm{be120.8.10}&120& 19049& {\bf\color{red}19035}& 407.5&{\bf\color{red}19035} &3.4 &19022 &578.0 \\
 \hline

 \textrm{be150.3.1}&150& 18889& {\bf\color{red}18889}& 533.9&{\bf\color{red}18889} &4.2 &18687 &767.9
&\textrm{be150.3.2}&150& 17816& {\bf\color{red}17816}& 553.7&{\bf\color{red}17816} &4.2 &17406 &760.5 \\
 \hline

 \textrm{be150.3.3}&150& 17314& {\bf\color{red}17314}& 550.0&{\bf\color{red}17314} &3.7 &17242 &781.0
&\textrm{be150.3.4}&150& 19884& 19878& 509.7&19878 &3.4 &{\bf\color{red}19884} &334.8 \\
 \hline

 \textrm{be150.3.5}&150& 16817& {\bf\color{red}16817}& 531.4&{\bf\color{red}16817} &3.6 &16714 &751.4
&\textrm{be150.3.6}&150& 16780& {\bf\color{red}16641}& 566.7&{\bf\color{red}16641} &3.8 &16457 &800.4 \\
 \hline

 \textrm{be150.3.7}&150& 18001& {\bf\color{red}18001}& 541.0&{\bf\color{red}18001} &3.7 &17813 &736.1
&\textrm{be150.3.8}&150& 18303& {\bf\color{red}18280}& 548.6&{\bf\color{red}18280} &4.0 &18069 &649.7 \\
 \hline

 \textrm{be150.3.9}&150& 12838& {\bf\color{red}12780}& 568.8&{\bf\color{red}12780} &4.1 &12265 &665.5
&\textrm{be150.3.10}&150& 17963& {\bf\color{red}17953}& 547.3&{\bf\color{red}17953} &3.9 &17724 &735.7 \\
 \hline

 \textrm{be150.8.1}&150& 27089& {\bf\color{red}27042}& 580.0&{\bf\color{red}27042} &4.1 &26450 &694.7
&\textrm{be150.8.2}&150& 26779& {\bf\color{red}26608}& 580.6&{\bf\color{red}26608} &3.8 &26288 &737.5 \\
 \hline

 \textrm{be150.8.3}&150& 29438& {\bf\color{red}29358}& 561.6&{\bf\color{red}29358} & 3.7&28896 &360.4
&\textrm{be150.8.4}&150& 26911& {\bf\color{red}26911}& 580.3&{\bf\color{red}26911} &3.8 &26366 &722.1 \\
 \hline

 \textrm{be150.8.5}&150& 28017& {\bf\color{red}27965}& 556.3&{\bf\color{red}27965} &3.9 &27869 &769.3
&\textrm{be150.8.6}&150& 29221& {\bf\color{red}29152}& 582.9&{\bf\color{red}29152} &3.6 &28631 &626.7 \\
 \hline

 \textrm{be150.8.7}&150& 31209& {\bf\color{red}31164}& 609.6&{\bf\color{red}31164} &3.8 &30751 &661.9
&\textrm{be150.8.8}&150& 29730& {\bf\color{red}29656}& 582.8&{\bf\color{red}29656} &4.1 &29147 &669.0 \\
 \hline

 \textrm{be150.8.9}&150& 25388& {\bf\color{red}25298}& 574.7&{\bf\color{red}25298} &4.0 &24904 &702.5
&\textrm{be150.8.10}&150& 28374& {\bf\color{red}28374}& 563.3&{\bf\color{red}28374} &3.8 &27885 &630.2 \\
 \hline

 \textrm{be200.3.1}&200& 25453& {\bf\color{red}25294}& 822.1&{\bf\color{red}25294} &4.4 &24123 &1199.4
&\textrm{be200.3.2}&200& 25027& {\bf\color{red}24983}& 825.8&{\bf\color{red}24983} &4.5 &24808 &1233.3 \\
 \hline

 \textrm{be200.3.3}&200& 28023& {\bf\color{red}27994}& 785.3&{\bf\color{red}27994} &4.7 &27585 &1299.9
&\textrm{be200.3.4}&200& 27434& {\bf\color{red}27363}& 823.4&{\bf\color{red}27363} &4.5 &27054 &1242.5 \\
 \hline

 \textrm{be200.3.5}&200& 26355& {\bf\color{red}26353}& 796.3&{\bf\color{red}26353} &4.8 &25390 &1214.3
&\textrm{be200.3.6}&200& 26146& {\bf\color{red}26138}& 785.3&{\bf\color{red}26138} &4.5 &25518 &1291.5 \\
 \hline

 \textrm{be200.3.7}&200& 30483& {\bf\color{red}30483}& 786.9&{\bf\color{red}30483} &4.1 &30086 &1232.7
&\textrm{be200.3.8}&200& 27355& {\bf\color{red}27287}& 775.0&{\bf\color{red}27287} & 4.6 &26944 &421.5 \\
 \hline

 \textrm{be200.3.9}&200& 24683& {\bf\color{red}24648}& 762.6&{\bf\color{red}24648} &4.8 &24172 &1227.3
&\textrm{be200.3.10}&200& 23842& {\bf\color{red}23708}& 806.0&23699 &4.7 &23307 &1255.9 \\
 \hline

 \textrm{be200.8.1}&200& 48534& {\bf\color{red}48419}& 835.6&{\bf\color{red}48419} &4.5 &47909 &1207.0
&\textrm{be200.8.2}&200& 40821& {\bf\color{red}40662}& 836.9&{\bf\color{red}40662} &4.6 &39043 &1242.3 \\
 \hline

 \textrm{be200.8.3}&200& 43207&{\bf\color{red} 43131}& 822.1&43095 &4.5 &41641 &1252.4
&\textrm{be200.8.4}&200& 43757& {\bf\color{red}43625}& 859.9&{\bf\color{red}43625} &4.3 &42796 &1257.0 \\
 \hline

 \textrm{be200.8.5}&200& 41482& {\bf\color{red}41214}& 793.5&{\bf\color{red}41214} &4.2 &40253 &1173.0
&\textrm{be200.8.6}&200& 49492& {\bf\color{red}49492}& 821.0&{\bf\color{red}49492} &4.3 &49382 & 1251.8\\
 \hline

 \textrm{be200.8.7}&200& 46828& {\bf\color{red}46813}& 846.6&{\bf\color{red}46813} &4.6 &46024 &1218.4
&\textrm{be200.8.8}&200& 44502& {\bf\color{red}44502}& 880.1&{\bf\color{red}44502} &4.3 &43208 &1212.1 \\
 \hline

 \textrm{be200.8.9}&200& 43241& {\bf\color{red}43241}& 825.7&{\bf\color{red}43241} &4.8 &42625 &1256.3
&\textrm{be200.8.10}&200& 42832& {\bf\color{red}42788}& 830.2&{\bf\color{red}42788} &4.6 &41594 & 1243.7\\
 \hline

 \textrm{be250.1}&250& 24076& {\bf\color{red}24067}& 985.5&{\bf\color{red}24067} &5.5 &23815 &2261.2
&\textrm{be250.2}&250& 22540& {\bf\color{red}22361}& 1026.2&{\bf\color{red}22361} &4.8 &22344 &2272.0 \\
 \hline

 \textrm{be250.3}&250& 22923& {\bf\color{red}22915}& 990.8&{\bf\color{red}22915} &5.4 &22783 &2162.3
&\textrm{be250.4}&250&24649 & {\bf\color{red}24610}& 980.3&{\bf\color{red}24610} &5.3 &24494 &2072.0 \\
 \hline

 \textrm{be250.5}&250& 21057& {\bf\color{red}21046}& 986.8&21040 &4.6 &20760 &2214.3
&\textrm{be250.6}&250& 22735& {\bf\color{red}22735}& 1019.2&{\bf\color{red}22735} &5.2 &22417 &2129.3 \\
 \hline

 \textrm{be250.7}&250& 24095& {\bf\color{red}24095}& 972.5&{\bf\color{red}24095} &5.0 &23888 &2148.0
&\textrm{be250.8}&250& 23801& {\bf\color{red}23709}& 991.5&{\bf\color{red}23709} &4.8 &23350 &1060.8 \\
 \hline

 \textrm{be250.9}&250& 20051& {\bf\color{red}19970}& 934.2&{\bf\color{red}19970} &4.9 &19729 &2212.4
&\textrm{be250.10}&250& 23159& {\bf\color{red}23077}& 1009.7&{\bf\color{red}23077} &4.8 &23009 &2102.1 \\
 \hline

 \textrm{gka8a}&100& 11109& 11101& 299.8&11101 &5.1 &{\bf\color{red}11109} &1.5
&\textrm{gka9b}&100& 137&0 & 376.0&0 &4.9 &{\bf\color{red}137} &246.2 \\
 \hline

 \textrm{gka10b}&125& 154&0 & 536.0&0 &4.8 &{\bf\color{red}154} &318.1
&\textrm{gka7c}&100& 7225&{\bf\color{red}7225} & 280.7&{\bf\color{red}7225} &4.2 &{\bf\color{red}7225} &3.6 \\
 \hline

 \textrm{gka1d}&100& 6333&6328 & 307.2&6328 &4.0 &{\bf\color{red}6333} &10.5
&\textrm{gka2d}&100& 6579&{\bf\color{red}6459} & 330.1&{\bf\color{red}6459} &3.7 &6446 &387.4 \\
 \hline

 \textrm{gka3d}&100& 9261&{\bf\color{red}9193} & 327.9&{\bf\color{red}9193} &3.8 &9179 &401.9
&\textrm{gka4d}&100& 10727&{\bf\color{red}10707} & 323.7&{\bf\color{red}10707} &4.0 &10695 &293.2 \\
 \hline

 \textrm{gka5d}&100& 11626&{\bf\color{red}11596} & 340.7&{\bf\color{red}11596} &3.6 &11410 &364.3
&\textrm{gka6d}&100& 14207&{\bf\color{red}14121} & 318.9&{\bf\color{red}14121} &4.0 &13598 &312.3 \\
 \hline

 \textrm{gka7d}&100& 14476&{\bf\color{red}14476} & 322.5&{\bf\color{red}14476} &3.5 &14325 &368.3
&\textrm{gka8d}&100& 16352&{\bf\color{red}16352} & 327.9&{\bf\color{red}16352} &3.8 &16254 &392.3 \\
 \hline

 \textrm{gka9d}&100& 15656&{\bf\color{red}15577} & 342.8&{\bf\color{red}15577} &3.8 &12631 &84.2
&\textrm{gka10d}&100& 19102&{\bf\color{red}19102} & 340.0&{\bf\color{red}19102} &3.8 &{\bf\color{red}19102} &362.9 \\
 \hline

 \textrm{gka1e}&200& 16464&16405 & 738.3&16405 &4.1 &{\bf\color{red}16431} &1390.9
&\textrm{gka2e}&200& 23395&{\bf\color{red}23360} & 791.9&{\bf\color{red}23360} &4.8 &23083 &1319.1 \\
 \hline

 \textrm{gka3e}&200& 25243&{\bf\color{red}25243} & 790.5&{\bf\color{red}25243} &4.2 &24444 &1228.8
&\textrm{gka4e}&200& 35594&{\bf\color{red}35559} & 802.2&{\bf\color{red}35559} &4.7 &35350 &1220.2 \\
 \hline

 \textrm{gka5e}&200& 35154&{\bf\color{red}35062} & 826.6&{\bf\color{red}35062} &4.4 &34145 &1119.9
& & & & & & & & & \\
 \hline
 \end{longtable}
 \end{center}}
 \vspace{-0.3cm}
 \subsection{Comparisons with dcSDPT3 for G-set instances}\label{sec5.3}

 Given a graph $\mathcal{G}=(\mathcal{V},\mathcal{E})$ with $|\mathcal{V}|=n$
 and a weight matrix $W\in\mathbb{S}^n$, the max-cut problem partitions $\mathcal{V}$
 into two nonempty sets $(\mathcal{Z}, \mathcal{V}\backslash\mathcal{Z})$ so that
 the total weights of the edges in the cut is maximized. It can be
 reformulated as \eqref{DC-BPP} with $p=n$ and $f(X)=\langle C,X\rangle$
 for $C=\frac{W-{\rm diag}(We)}{4}$. We solve the G-set instances with $W$
 from \url{http://www.stanford.edu/yyye/yyye/Gset} for $800$ to $20000$ variables.
 Table \ref{table2} reports the relative gap and infeasibility of the outputs of
 dcFAC and dcSDPT3 and the CPU time taken, where ``-'' means that the CPU time is
 more than \textbf{2} hours. Since it is impractical for an exact method,
 say BiqCrunch \cite{Krislock17}, to yield optimal values for these instances,
 Table \ref{table1} lists the known best values got with some advanced heuristic methods \cite{Wu15,Shylo15}.
 \setlength{\tabcolsep}{0.5mm}{
 \begin{center}
 \setlength{\abovecaptionskip}{0pt}%
 \setlength{\belowcaptionskip}{2pt}%
 \captionsetup{font={scriptsize}}
 \renewcommand\arraystretch{1.5}
  \scriptsize
 \begin{longtable}{|c|c|ccc|ccc||c|c|ccc|ccc|}
  \caption{Numerical results of dcFAC and dcSDPT3 for the G-set instances}\label{table2}\\
  \hline
  & &\multicolumn{3}{c|}{\bf dcSDPT3}&\multicolumn{3}{c||}{\bf dcFAC}& & &\multicolumn{3}{c|}{\bf dcSDPT3}
  &\multicolumn{3}{c|}{\bf dcFAC}\\
  \cline{3-8}\cline{11-16}
  Name($n$)&Bval&\!gap($\%$)\!&\!time\!&\!infeas\!&\!gap($\%$)\!&\!time\!&\!infeas\!
  &Name($n$)&Bval&\!gap($\%$)\!&\!time\!&\!infeas\!&\!gap($\%$)\!&\!time\!&\!infeas\\
  \hline
   G1(800)&11624&0.052&442.7&2.6e-13&{\bf\color{red}0.017}&7.9&4.0e-10
  &G2(800)&11620&0.129&467.3&3.1e-13&{\bf\color{red}0.069}&7.2&5.8e-10\\
 \hline

  G3(800)&11622 & 0.146& 442.0 & 3.1e-12&{\bf\color{red}0.121}&7.2&1.5e-9
 &G4(800)&11646 & 0.232& 438.2 & 3.8e-13&{\bf\color{red}0.112}&6.9&2.5e-10\\
 \hline

  G5(800)& 11631 & {\bf\color{red}0.146}& 441.5 & 5.2e-13&{\bf\color{red}0.146}&6.4&3.3e-10
 &G6(800)& 2178  & 1.240& 478.4 &6.3e-13&\!{\bf\color{red}0.321}\!&\!6.3\!&\!1.9e-10\\
 \hline

  G7(800)& 2006  &1.645& 512.4 &5.3e-13&{\bf\color{red}1.047}&6.5&2.2e-10
 &G8(800)& 2005  &1.845& 471.6 &2.7e-13&{\bf\color{red} 1.197}&6.7\!&\!2.7e-10\\
 \hline

  G9(800)& 2054  & 1.412& 462.3 &6.2e-13&\!{\bf\color{red}1.169}\!&\!6.4\!&\!4.2e-10
 &G10(800)&2000  & 1.200& 483.4 &4.5e-13 & {\bf\color{red}0.600}&4.8 &3.4e-10\\
 \hline

  G11(800)& 564 &{\bf\color{red} 2.128}& 221.0 &3.7e-11 &{\bf\color{red} 2.128}&6.5 &8.9e-10
 &G12(800)& 556 & 2.518& 220.3 &9.3e-12 &{\bf\color{red} 1.439}&6.7 &7.4e-10\\
 \hline

  G13(800)& 582 & 1.031& 216.7 &3.4e-13 & {\bf\color{red} 0.687}&6.8 &5.5e-10
 &G14(800)& 3064& 0.783& 255.5 & 1.1e-11 &{\bf\color{red} 0.490} &6.8 &1.6e-9\\
 \hline

 G15(800)& 3050  & 0.623&250.9& 9.3e-12&{\bf\color{red} 0.557}&6.5 &3.3e-9
 &G16(800)& 3052 &0.590&261.5 &7.4e-11 &{\bf\color{red} 0.590}&8.6 &3.4e-9\\
 \hline

 G17(800)& 3047  & 0.295&210.4 &1.1e-11& {\bf\color{red}0.230}&6.4 &2.3e-9
 &G18(800)& 992  & 1.613&330.5 & 1.9e-12 &{\bf\color{red}1.411}&8.9 &1.7e-9\\
 \hline

 G19(800)& 906   & 3.091&393.3 &6.0e-12 &{\bf\color{red} 2.870}&8.9 &2.0e-9
 &G20(800)& 941  & 1.594&349.5 & 1.0e-11 & {\bf\color{red}1.169}&7.5 & 1.2e-9\\
 \hline

 G21(800)& 931 &2.578& 349.6 & 5.7e-12 & {\bf\color{red}2.256}&7.3 &1.6e-9
 &G22(2000)& 13359 & 0.397&3115.9 & 1.7e-12 &{\bf\color{red} 0.210}&18.4 &6.3e-10\\
 \hline

 G23(2000)& 13344 & 0.427&3147.8 & 2.6e-12 &{\bf\color{red} 0.345}&18.4 &4.9e-10
 &G24(2000) & 13337 & 0.300&2977.3 & 2.7e-12 & {\bf\color{red} 0.285}&18.4 &2.4e-10\\
 \hline

 G25(2000) & 13340 & 0.502&3077.2 & 1.2e-12 & {\bf\color{red} 0.315}&19.1 &1.3e-9
 &G26(2000)& 13328 & 0.398&2992.7 & 2.4e-12 & {\bf\color{red} 0.165}&18.6 &3.1e-10\\
 \hline

 G27(2000) & 3341  & 1.137&3271.5 & 4.1e-13 & {\bf\color{red} 0.838}&18.9 &1.0e-9
 &G28(2000) & 3298  & 1.152&3211.4 &4.8e-13 &{\bf\color{red} 0.303}&18.5 &9.3e-10\\
 \hline

 G29(2000) & 3405 & 1.234&3113.3& 3.9e-13 &{\bf\color{red} 1.057}&18.3 &7.9e-10
 &G30(2000) & 3413  & 1.143&3413.7 & 3.0e-12 & {\bf\color{red} 1.055}&18.4 &4.6e-10\\
 \hline

 G31(2000) & 3310  & 1.692&3388.2& 5.7e-13 & {\bf\color{red} 0.665}&18.7 &3.5e-10
 &G32(2000) & 1410  & 2.553&1618.7 & 3.6e-12 & {\bf\color{red} 1.277}&17.9 & 7.9e-10\\
 \hline

 G33(2000) & 1382  & 2.316 &1811.0& 5.4e-13 &{\bf\color{red}1.447}&17.2 & 1.0e-9
 &G34(2000) & 1384  &1.734 &1650.6 & 4.5e-13 & {\bf\color{red} 1.301}&17.6 &6.7e-10\\
 \hline

 G35(2000) & 7687  & 0.650&2137.9 & 4.4e-11 & {\bf\color{red} 0.455}&18.1 &2.1e-9
 &G36(2000) & 7680  &0.651&2356.0 &3.3e-12 & {\bf\color{red}0.430}&16.8 &3.3e-9\\
 \hline

 G37(2000)& 7691 & 0.572&2294.6 & 2.5e-12 & {\bf\color{red} 0.468}&18.7 & 2.1e-9
 &G38(2000) & 7688 &0.820&2823.5 & 6.3e-13 & {\bf\color{red} 0.650}&21.2 &3.0e-10\\
 \hline

 G39(2000)& 2408 & 2.533&3073.3 & 3.6e-12 &{\bf\color{red} 1.827}&20.4 &1.4e-9
 &G40(2000)& 2400 &2.375&2986.1 & 4.3e-12 &{\bf\color{red} 1.458}&22.4 & 2.4e-9\\
 \hline

 G41(2000) & 2405 & 1.746&2687.7 & 2.0-10 & {\bf\color{red}0.707}&24.6 &1.2e-9
 &G42(2000) & 2481 &3.023&3684.9 & 5.8e-12 & {\bf\color{red} 1.854}&24.3 &1.5e-9\\
 \hline

 G43(1000) & 6660 & 0.210&594.1 & 3.7e-13 & {\bf\color{red} 0.090}&7.9 &3.3e-10
 &G44(1000) & 6650 & 0.135&636.7 & 2.2e-13 &{\bf\color{red} 0.105}&8.2 &4.0e-10\\
 \hline

 G45(1000) & 6654 & 0.586&640.6 & 2.3e-13 & {\bf\color{red} 0.256}&8.3 &3.4e-10
 &G46(1000) & 6649 & 0.241&597.1 & 1.3e-12 & {\bf\color{red}0.271}&8.2&3.0e-10\\
 \hline

 G47(1000)& 6657 & 0.300&592.4 & 1.1e-12 & {\bf\color{red} 0.210}&8.3 &2.4e-10
 &G48(3000)&6000& {\bf\color{red}0 } &148.4 & 6.9e-14& {\bf\color{red}0 }&13.5 &2.0e-12\\
 \hline

 G49(3000)& 6000 &{\bf\color{red}0 } &142.6 & 8.0e-14 &  {\bf\color{red}0}&13.5 &3.2e-12
 &G50(3000) & 5880 &{\bf\color{red}0 }&3095.6 &1.1e-12 &  {\bf\color{red}0}&25.9 &1.5e-11\\
 \hline

 G51(1000) & 3848 &0.702&402.7 & 7.5e-12 &{\bf\color{red} 0.676}&8.3 &3.0e-9
 &G52(1000) & 3851 & 0.571&430.5 & 1.3e-12 & {\bf\color{red}0.545}&11.2 &3.5e-9 \\
 \hline

 G53(1000) & 3850 & {\bf\color{red}0.571}&343.5 & 1.6e-10 & 0.623&10.3 &3.1e-9
 &G54(1000)  & 3852 & 0.649&378.9 & 7.7e-11 &{\bf\color{red} 0.441}&9.6 &2.9e-9\\
 \hline

 G55(5000)& 10299 & -& - &- & 0.437& 73.2 &1.5e-9
 &G56(5000) & 4017 & -& - &- & 1.120&73.3 &1.3e-9\\
 \hline

 G57(5000) & 3494 & -&- &- & 1.431&77.5 &6.2e-10
 &G58(5000) & 19293 & -& - &- & 0.549&74.2 &3.1e-9\\
 \hline

 G59(5000) & 6086 & -& - &- &2.021&97.0 & 1.3e-9
 &G60(7000) & 14188 & -& - &- &0.585&120.5 & 1.2e-9\\
 \hline

 G61(7000) & 5796 & -& - &- & 1.346 &123.3 &9.7e-10
 &G62(7000) & 4870 & -& - &- & 1.602&137.2 & 6.2e-10\\
 \hline

 G63(7000) & 27045 & -& - &- &0.669&139.0 & 3.8e-9
 &G64(7000) & 8751 & -& -  &- & 2.080&161.8 &2.4e-9\\
 \hline

 G65(8000)& 5562 & -& -  &- & 1.654&172.7 & 5.3e-10
 &G66(9000) & 6364 & -& -  &- & 1.917&218.2 &3.0e-10\\
 \hline

 G67(10000)& 6950 & -& -  &- & 1.496&270.5 & 1.7e-10
 & G70(10000) & 9591 & -& - &- & 0.250&216.8 & 2.6e-9\\
 \hline

 G72(10000) & 7006 & -& - &- &1.827&259.6 & 1.7e-10
 & G77(14000) & 9938 & -& - &- &1.872&471.6 & 7.5e-10 \\
\hline
 G81(20000) & 14048 & -& - &- &1.751&\!879.6\!&\!9.6e-10
 &  &  & &  & &&&\\
 \hline
 \end{longtable}
 \end{center}}

 We see that the outputs of dcFAC have the least gap for almost all instances,
 though their infeasibility is a little worse than that of dcSDPT3.
 The relative gaps of the outputs for dcFAC and dcSDPT3 are respectively
 at most \textbf{2.870\%} and \textbf{3.091\%}. When $n=5000$,
 the CPU time taken by dcSDPT3 is more than $2$ hours, but dcFAC yields
 the desirable result for the instance with $n=20000$ in $900$ seconds.
 By comparing the results of dcSDPT3 with those in Table \ref{table1},
 we conclude that the use of $\rho_0=0.1,\sigma=1.05$ leads to its worse performance.
 We also compare the relative gaps of dcFAC with the relative gaps
 for the rounding of its outputs, and find that their maximal error is \textbf{8.20e-9}.
 This means that the rounding of the final output has little influence on
 the objective value if the infeasibility is in the magnitude of $10^{-9}$.
 \subsection{Comparisons with dcSDPT3 for OR-Library instances}\label{sec5.4}

 This part compares the performance of dcFAC with that of dcSDPT3 for
 $\max_{z\in\{0,1\}^n}z^{\mathbb{T}}Az$, with $A\in\mathbb{S}^n$
 from the OR-Library \url{http://people.brunel.ac.uk/~mastjjb/jeb/orlib/bqpinfo.html}.
 Table \ref{table3} reports the relative gap and infeasibility of their outputs,
 the CPU time taken, and the known best values obtained in \cite{Palubeckis04}
 with an advanced heuristic method. We see that the relative gaps of the outputs
 for dcFAC and dcSDPT3 are respectively not more than \textbf{0.688\%} and \textbf{0.671\%},
 and the outputs of dcFAC have the less relative gap for most instances with
 infeasibility less than $10^{-9}$. When $n=2500$, dcFAC yields the desired
 result in \textbf{40}s but dcSDPT3 can not yield the result in \textbf{2}h.
 \setlength{\tabcolsep}{0.6mm}{
 \begin{center}
  \setlength{\abovecaptionskip}{0pt}%
 \setlength{\belowcaptionskip}{2pt}%
 \captionsetup{font={scriptsize}}
 \renewcommand\arraystretch{1.5}
 \scriptsize
 \begin{longtable}{|c|c|ccc|ccc||c|c|ccc|ccc|}
  \caption{Numerical results of dcSDPT3 and dcFAC for the OR-Library instances}\label{table3}\\
  \hline
  &\!&\multicolumn{3}{c|}{\bf dcSDPT3}&\multicolumn{3}{c||}{\bf dcFAC}&\!& &
  \multicolumn{3}{c|}{\bf dcSDPT3}&\multicolumn{3}{c|}{\bf dcFAC}\\
  \cline{3-8} \cline{11-16}
  Name&Bval&\!gap($\%$)\!&\!time\!&\!infeas\!&\!gap($\%$)\!&\!time\!&\!infeas\!&Name&Bval
  &\!gap($\%$)\!&\!time\!&\!infeas\!&\!gap($\%$)\!&\!time\! &\!infeas\\
  \hline
 \textrm{1000\_1}&371438& 0.257& 1403.5 & 1.1e-11 &{\bf\color{red}0.173}& 11.8& 5.8e-10
 &\textrm{1000\_2}&354932&0.395& 1472.0 & 4.0e-12&{\bf\color{red}0.291}& 11.0& 2.7e-10\\
 \hline
 \textrm{1000\_3} & 371236 &0.255& 1494.8 & 1.2e-11 & {\bf\color{red}0.233}&11.1 & 5.1e-10
 &\textrm{1000\_4} & 370675 &0.235& 1481.0 & 3.4e-12 &{\bf\color{red}0.157}&11.1 & 8.4e-10\\
  \hline
 \textrm{1000\_5} & 352760 &0.339& 1473.3  & 1.5e-11 &{\bf\color{red}0.149}&11.0 & 1.0e-9
 &\textrm{1000\_6} & 359629 &0.439&1431.8 & 8.3e-12 &{\bf\color{red}0.432}&11.0& 2.3e-10\\
  \hline
 \textrm{1000\_7} & 371193 &{\bf\color{red}0.671}& 1464.1 & 7.4e-12 &0.688&11.0 & 6.1e-10
 &\textrm{1000\_8} & 351994 &0.388& 1462.4 &4.3e-12 &{\bf\color{red}0.354}&11.0& 2.7e-10\\
  \hline
 \textrm{1000\_9} & 349337 &0.192& 1478.3 &4.2e-12 &{\bf\color{red}0.034}&10.7 & 2.5e-10
 &\textrm{1000\_10} & 351415 & 0.378&1414.9&7.1e-12 &{\bf\color{red}0.191}&10.9 & 7.9e-10\\
 \hline
 \textrm{2500\_1} & 1515944 & -& -& -& 0.249 &33.3 & 1.2e-10
 &\textrm{2500\_2} & 1471392 & -& -& -& 0.172 &33.1 &2.7e-10\\
  \hline
 \textrm{2500\_3} & 1414192 & -& -& -& 0.338 &32.0 & 1.6e-10
 &\textrm{2500\_4} & 1507701 & -& -& -& 0.216 &32.5 & 2.6e-10\\
 \hline
 \textrm{2500\_5} & 1491816 & -& -& -& 0.184 &31.7 & 1.5e-10
 &\textrm{2500\_6} & 1469162 & -& -& -& 0.222 &31.3 & 3.1e-10\\
 \hline
 \textrm{2500\_7} & 1479040 & -& -& -& 0.377 &32.2& 1.3e-10
 &\textrm{2500\_8} & 1484199 & -& -& -& 0.147 &33.1& 9.8e-10\\
 \hline
 \textrm{2500\_9} & 1482413 & -& -& -& 0.275 &34.3 & 1.7e-10
 &\textrm{2500\_10} & 1483355 & -& -& -& 0.307 &31.9 & 1.5e-10\\
 \hline
 \end{longtable}
 \end{center}}
  \vspace{-0.3cm}
 \subsection{Numerical results of dcFAC for Palubeckis instances}\label{sec5.5}

 This part provides the results of dcFAC for solving
 $\max_{z\in\{0,1\}^n}z^{\mathbb{T}}Az$ with $A\in\mathbb{S}^n$ from
 the Palubeckis instances \url{https://www.personalas.ktu.lt/~ginpalu/},
 and the known best value obtained in \cite{Glover10} with an advanced
 heuristic method. Since these instances involve more than $3000$ variables,
 and it is time consuming for dcSDPT3 to compute an instance, we do not compare
 the results of dcFAC with those of dcSDPT3. From Table \ref{table4}, the outputs
 of dcFAC have the relative gaps at most \textbf{0.356\%} for the $21$ instances.
 \setlength{\tabcolsep}{1.5mm}{
 \begin{center}
 \setlength{\abovecaptionskip}{0pt}%
 \setlength{\belowcaptionskip}{2pt}%
 \captionsetup{font={scriptsize}}
 \renewcommand\arraystretch{1.5}
 \scriptsize
 \begin{longtable}{|c|c|ccc||c|c|ccc|}
 \caption{Numerical results of dcFAC for the Palubeckis instances}\label{table4}\\
 \hline
  &\!&\multicolumn{3}{c||}{\bf dcFAC}& &\!&\multicolumn{3}{c|}{\bf dcFAC}\\
  \cline{3-5} \cline{8-10}
  Instance&Bval&\!gap($\%$) &time(s) &infeas &Instance&Bval&gap($\%$) &time(s) &infeas \\
 \hline

 p3000.1 & 3931583 & 0.347 &53.1& 2.9e-10& p3000.2 & 5193073 & 0.233 & 53.8& 7.5e-11\\
 \hline

 p3000.3 & 5111533 &0.320 & 53.2& 1.6e-10& p3000.4 & 5761822 &0.283 & 53.5& 1.4e-10\\
 \hline

 p3000.5 & 5675625 & 0.308 & 53.5& 1.5e-10& p4000.1 & 6181830 &0.253 & 80.7& 8.5e-11\\
 \hline

 p4000.2 & 7801355 &0.345 & 83.3& 6.2e-11& p4000.3 & 7741685 &0.354 & 84.5& 7.2e-11\\
 \hline

 p4000.4 & 8711822 &0.310 & 83.5& 2.5e-11& p4000.5 & 8908979 &0.356 & 82.1& 6.0e-11\\
 \hline

 p5000.1 & 8559680 &0.308 & 115.0& 6.1e-11& p5000.2 & 10836019 &0.306 & 119.2& 1.2e-10\\
 \hline

 p5000.3 & 10489137 &0.296 & 121.2&1.1e-10& p5000.4 & 12252318 &0.274 & 120.2&1.2e-10\\
 \hline

 p5000.5 & 12731803 &0.304 & 120.2&1.1e-10& p6000.1 & 11384976 &0.273 & 159.4& 5.7e-11\\
 \hline

 p6000.2 & 14333855 &0.202 & 161.3& 7.8e-11& p6000.3 & 16132915 &0.336 & 161.5&8.7e-11\\
 \hline

 p7000.1 & 14478676 &0.272 & 202.7& 5.7e-11& p7000.2 & 18249948 &0.253 & 205.1&7.2e-11\\
 \hline

 p7000.3 & 20446407 &0.302 & 205.3& 5.3e-11& & & & & \\
 \hline
\end{longtable}
\end{center}}

\vspace{-0.3cm}
 \subsection{Comparisons with dcSNCG for UBPP instances}\label{sec5.6}

 With $X=(1;x_1;\ldots;x_q)(1;x_1;\ldots;x_q)^{\mathbb{T}}$, we can reformulate
 \eqref{UBPP} with $\vartheta$ from \eqref{vtheta} as \eqref{DC-BPP} for
 $f(X)=\prod_{i=1}^q\langle C_i,X\rangle$ with
 $C_i=\left(\begin{matrix}
          a_i & b_i^{\mathbb{T}} \\
          b_i & B_i  \\
        \end{matrix}\right)$,
 where $b_i\!=\!({\bf 0}_{n(i-1)};\frac{c_i}{2};{\bf 0}_{n(q-i)})\!\in\mathbb{R}^{nq}$
 and $B_i={\rm BlkDiag}(0,\ldots,$
 
 \noindent
 $0,Q_i,0,\ldots,0)\in\!\mathbb{S}^{nq}$ for $i=1,\ldots,q$.
 We test the performance of dcFAC and dcSNCG for solving this class of examples with $q=2$.
 To verify the efficiency of Algorithm \ref{Alg2} with varying $L_k$,
 we compare their performance with that of Algorithm \ref{Alg-factor}
 armed with \cite[Algorithm 2]{LiuPong191} (dcFAC\_ls for short), where
 Algorithm 2 of \cite{LiuPong191} is an MM method with linesearch technique
 for solving \eqref{epenalty-factor}. In addition, we also compare their performance
 with that of GloptiPoly3 \cite{Henrion09}, a software for the Lasserre relaxation of polynomial programs.
 Let $\mathbb{N}^k_r\!:=\{\alpha\in\mathbb{N}^k\,|\,\sum_{i=1}^k\alpha_i\le r\}$
 and $s(r)\!:=\!(\begin{smallmatrix} qn+r \\ qn \end{smallmatrix})$ for $k,r\in\mathbb{N}$.
 The $r(r\ge q)$-order Lasserre relaxation of \eqref{UBPP} is given by
 \[
   \inf_{y}\bigg\{\sum_{\alpha\in\mathbb{N}^{qn}_{2r}}p_{\alpha}y_{\alpha}\ \ {\rm s.t.}\
    M_{r}(y)\in\mathbb{S}_{+}^{s(r)},M_{r-1}(h_iy)=0,\ i=1,\ldots,nq\bigg\},
 \]
 where $p_{\alpha}$ is the component of the coefficient vector of $\vartheta(x)$,
 $h_i$ is the coefficient of $h_i(x)=x_i^2-1$ for $i=1,2,\ldots,qn$, and $M_{r}(y)$
 and $M_{r-1}(h_iy)$ are respectively the moment matrix of dimensions
 $s(r)$ and $s(r\!-\!1)$ (see \cite{Lasserre01} for the details).

 The first group of problems is using $Q_i\!=\!\frac{1}{4}D_i$
 $c_i\!=\!2Q_ie$ and $a_i\!=\!e^{\mathbb{T}}Q_ie+\omega_i$ with $D_i=\frac{M_i}{\|M_i\|}$
 for $i=1,2$, where the entries of each $M_i\in\mathbb{S}^l$ and $\omega_i$ are generated
 to obey the standard normal distribution. Such a problem is a reformulation of
 \begin{equation}\label{polyprob1}
  \max_{x,y\in\{0,1\}^{l}}\Big\{-(x^{\mathbb{T}}D_1x+\omega_1)(y^{\mathbb{T}}D_2y+\omega_2)\Big\}.
 \end{equation}
 Table \ref{table5} reports the results of three solvers for solving \eqref{polyprob1}
 with different $l$ and those of GloptiPoly3 for solving its $2$-order Lasserre relaxation.
 For $l=2$ and $8$, the three solvers deliver the same objective value as GloptiPoly3 does,
 which now becomes the optimal since the Lasserre relaxation provides an upper bound for the optimal value.
 For $l\ge 20$, GloptiPoly3 fails to deliver the result due to out of memory, while dcFAC,
 dcFAC\_ls and dcSNCG can provide an approximate upper bound of the optimal value even for $l=1000$
 within $127s, 136s$ and $6276s$, respectively. For \textbf{12} instances, the outputs of dcFAC
 and dcFAC\_ls respectively have \textbf{8} and \textbf{6} best objective values, and dcFAC requires
 more CPU time than dcFAC\_ls does due to the worse $L_0$.

 \medskip

\setlength{\tabcolsep}{1.5mm}{
 \begin{center}
 \setlength{\abovecaptionskip}{0pt}%
 \setlength{\belowcaptionskip}{2pt}%
 \captionsetup{font={scriptsize}}
 \renewcommand\arraystretch{1.5}
 \scriptsize
 \begin{longtable}{|c|c|ccc|ccc|ccc|}
 \caption{Numerical results of dcFAC, dcFAC\_ls, GloptiPoly3 and dcSNCG for \eqref{polyprob1}}\label{table5}\\
 \hline
 \raisebox{2.0ex}[15pt]{}
  &\multicolumn{1}{c|}{GPoly3}&\multicolumn{3}{c|}{\bf dcSNCG}&\multicolumn{3}{c|}{\bf dcFAC}
  &\multicolumn{3}{c|}{\bf dcFAC\_ls}\\
  \hline
  $l$ &\! obj &\!obj &time &infeas&\!obj &time &infeas&\!obj&time &infeas\\
  \hline
  2 & 5.6696  & 5.6692&0.1& 2.1e-5&{\bf\color{red}5.6696}& 0.1& 2.2e-9& {\bf\color{red}5.6696}& 0.1&2.4e-9\\
  \hline
  8 & 22.3922 & {\bf\color{red}22.3923}&0.6& 1.8e-5&{\bf\color{red}22.3922} & 1.7& 1.5e-9& {\bf\color{red}22.3922}& 5.6&1.5e-9\\
  \hline
  20 &* &{\bf\color{red}58.1403}&1.3& 4.7e-6&57.8066 & 6.0& 1.1e-9&57.8066 &28.6 &1.0e-9\\
  \hline
  100 & * & 683.3912&16.4& 1.3e-5 &{\bf\color{red}705.0149}& 21.7& 8.1e-10&{\bf\color{red}705.0149} &16.81 &1.0e-10\\
  \hline
  200 & * & 3.6612e+3&97.9& 1.8e-5 & 3.7713e+3& 39.3& 2.3e-11&{\bf\color{red}3.7724e+3} &22.4 &9.1e-11\\
  \hline
  300 & * & 7.7539e+3&234.9& 1.2e-5 &{\bf\color{red}7.9862e+3}& 58.8& 3.9e-12&7.9800e+3 &28.7 &2.7e-11\\
  \hline
  500 & * & 2.3830e+4&881.7& 1.8e-6 &2.4436e+4& 131.5&1.5e-12&{\bf\color{red}2.4441e+4} &55.2 &7.0e-12\\
  \hline
  800 & * & 5.1182e+4&3290.9& 5.5e-5 &{\bf\color{red}5.3332e+4}& 274.9&1.6e-13&5.3248e+4 &100.7 &1.6e-12\\
  \hline
  1000& * & 8.3766e+4&6275.7& 2.6e-5 &{\bf\color{red}8.7175e+4}& 426.5&5.1e-14&8.6954e+4 &135.8 &3.4e-12\\
  \hline
  1200& * &-&-& -&{\bf\color{red}1.3474e+5}& 595.6&3.5e-14&1.3455e+5 &180.5 &1.6e-12\\
  \hline
  1500& * & -&-& -&{\bf\color{red}2.0753e+5}& 1019.2&1.8e-14&2.0736e+5 &269.0 &8.8e-13\\
  \hline
  2000& * & -&-& -&3.6092e+5& 1717.2&1.2e-14&{\bf\color{red}3.6116e+5} &459.1 &9.1e-13\\
  \hline
 \end{longtable}
 \end{center}
 }

 The second group of problems is using $Q_i\!=\!\frac{1}{4}(-1)^i\overline{W}^i,
 a_i\!=\frac{1}{4}(-1)^{i+1}e^{\mathbb{T}}\overline{W}^ie$ and $c_i=0$ with
 $\overline{W}^i=\frac{W^i}{\|W^i\|}$ for $i=1,2$, where each
 $W^i\in\mathbb{S}^l$ is chosen from the G-set and the Biq Mac Library.
 Such a problem is a reformulation of the generalized max-cut problem
 \begin{equation}\label{polyprob2}
  \max_{x,y\in\{-1,1\}^l}\frac{1}{4}\bigg\{\Big[\sum_{i<j}\overline{w}_{ij}^{1}(1-x_ix_j)\Big]
  \Big[\sum_{i<j}\overline{w}_{ij}^2(1-y_iy_j)\Big]\bigg\}.
 \end{equation}
 Table \ref{table6} reports the results of dcFAC, dcFAC\_ls and dcSNCG
 for solving problem \eqref{polyprob2} with different $(W^{1},W^2)$.
 Among \textbf{14} instances, the outputs of dcFAC and dcFAC\_ls respectively
 have \textbf{8} and \textbf{7} best objective values, and dcFAC requires
 a little more CUP time than dcFAC\_ls does. When $n=2000$, dcFAC\_ls and dcFAC
 can yield the result within $1500s$, but dcSNCG can not yield
 the result within \textbf{2} hours.
 \setlength{\tabcolsep}{1.5mm}{
 \begin{center}
  \setlength{\abovecaptionskip}{0pt}%
 \setlength{\belowcaptionskip}{2pt}%
 \captionsetup{font={scriptsize}}
 \renewcommand\arraystretch{1.5}
 \scriptsize
 \begin{longtable}{|c|c|ccc|ccc|ccc|}
  \caption{Numerical results of dcFAC, dcFAC\_ls and dcSNCG for problem \eqref{polyprob2}}\label{table6}\\
  \hline
  &\!&\multicolumn{3}{c|}{\bf dcSNCG}&\multicolumn{3}{c|}{\bf dcFAC}&\multicolumn{3}{c|}{\bf dcFAC\_ls}\\
  \hline
  $(\overline{W}^{1},\overline{W}^2)$&n&\! Obj & time &infeas & Obj &time &infeas & Obj &time &infeas\\
 \hline

 \textrm{ising2.5\_200(5,6)}&200& 8.2944e+3& 95.4& 1.7e-5 &8.4516e+3  & 93.5 & 2.3e-9&{\bf\color{red}8.4686e+3} &248.4 &2.3e-9\\
 \hline

 \textrm{ising3.0\_200(5,6)}&200&8.1765e+3& 93.6& 1.9e-6 &{\bf\color{red}8.2312e+3}  &93.0 & 2.9e-9&{\bf\color{red}8.2312e+3} &148.9 &2.8e-9\\
 \hline

 \textrm{ising2.5\_300(5,6)}&300& 2.0570e+4 & 245.0& 3.4e-5 &{\bf\color{red}2.0933e+4} & 115.4 & 1.8e-9&2.0877e+4 &115.7 &1.6e-9\\
 \hline

 \textrm{ising3.0\_300(5,6)}&300& 1.9571e+4 & 237.6& 3.7e-5 &{\bf\color{red}1.9688e+4} & 139.3 & 1.9e-9&1.9682e+4 &165.1 &2.2e-9\\
 \hline

 \textrm{t2g20(5,6)}&400&5.5053e+4 & 507.9& 1.0e-5 & {\bf\color{red}5.7777e+4}& 170.4 & 1.0e-9&5.7378e+4 &126.4 &4.2e-10\\
 \hline

 \textrm{t2g20(6,7)}&400&5.4338e+4 & 513.1 &2.1e-5 &{\bf\color{red}5.9303e+4}& 196.2 & 1.1e-9&5.9191e+4 &136.9 &1.0e-9\\
 \hline

 \textrm{t2g20(5,7)}&400&5.5202e+4 & 502.8& 1.7e-5 &5.7637e+4& 179.1 & 4.0e-10& {\bf\color{red}5.7698e+4}&152.8 &5.0e-10\\
 \hline

 \textrm{(G7,G8)}&800&3.2099e+5 & 3289.1& 5.0e-6 &{\bf\color{red}3.6194e+5}& 330.1 & 3.7e-14&3.6178e+5&138.6 &1.1e-13\\
 \hline

 \textrm{(G8,G9)}&800& 3.2154e+5 & 3203.1& 1.1e-6 &3.5405e+5& 333.2 & 2.7e-14&{\bf\color{red}3.5867e+5} &139.8 &4.1e-13\\
 \hline

 \textrm{(G9,G10)}&800&3.3605e+5 & 3172.1& 2.7e-5 &{\bf\color{red}3.5967e+5}& 319.8 & 5.2e-14& 3.5831e+5&130.4 &1.7e-13\\
 \hline

 \textrm{(G43,G44)}&1000& 1.5799e+5 & 7114.4& 6.9e-5&2.9666e+5& 692.2 & 4.4e-16&{\bf\color{red}3.0460e+5} &869.6 &3.2e-9\\
 \hline

 \textrm{(G45,G46)}&1000&1.6003e+5 & 7355.5& 2.4e-5 &2.9288e+5& 693.2 & 4.4e-16&{\bf\color{red}3.0162e+5} &930.6 &1.4e-9\\
 \hline

 \textrm{(G31,G32)}&2000& - & -&- &2.2827e+6 & 1366.6 & 4.3e-14&{\bf\color{red}2.2852e+6} & 596.4&2.9e-13\\
  \hline

 \textrm{(G33,G34)}&2000& - & -& -& {\bf\color{red}2.4285e+6}& 1377.6 & 1.2e-15&2.4180e+6 & 1127.9&3.0e-13\\
  \hline
 \end{longtable}
 \end{center}}
 \section{Conclusions}\label{sec7}

 We have proposed a relaxation approach to the UBPP \eqref{UBPP} by seeking
 a finite number of stationary points of the DC penalized matrix program
 \eqref{BPP-factor} with increasing penalty factors, and developed a globally
 convergent MM method with extrapolation to achieve such stationary points.
 The rank-one projections of its outputs are shown to be approximate feasible
 to the UBPP under a mild condition, and the upper bound of their objective values
 to the optimal value is also quantified. Numerical comparisons with SDPRR
 for \textbf{119} Biq Mac Library instances and with dcSNCG for \textbf{26}
 UBPP instances constructed with $q=2$ show that dcFAC is remarkably superior to
 SDPRR and dcSNCG by the quality of the output and the CPU time.
 The comparisons with dcSDPT3 for \textbf{119} Biq Mac Library instances and
 \textbf{112} UBQP instances indicate that dcFAC is comparable with dcSDPT3
 if the latter is using the same updating rule of $\rho$ (only possible for small-scale instances),
 otherwise is superior to dcSDPT3 in the quality of solutions and the CPU time.

\bigskip
\noindent
{\large\bf Appendix A: The proof of Proposition \ref{prop-Alg2}}

 \begin{proof}
 {\bf(i)} From the definition of $V^{k+1}$ and the feasibility of $V^k$ to problem \eqref{Vk-subprob},
 it follows that
 \begin{align}\label{key-ineq0}
  &\langle\nabla\!\widetilde{f}(U^k)\!+\!\rho \Gamma^{k},V^{k+1}\rangle
   \!+\!\rho\|V^{k+1}\|_F^2\!+\!(L_k/2)\|V^{k+1}\!-\!U^k\|_F^2\nonumber\\
  &\le \langle \nabla\!\widetilde{f}(U^k)\!+\!\rho \Gamma^{k},V^{k}\rangle
     \!+\!\rho\|V^{k}\|_F^2\!+\!(L_k/2)\|V^{k}\!-\!U^k\|_F^2.
 \end{align}
 Notice that $\Gamma^k\in\partial\widetilde{\psi}(V^k)\subseteq-\partial(-\widetilde{\psi})(V^k)$.
 From the convexity of $-\widetilde{\psi}$ and \cite[Theorem 23.5]{Roc70},
 $\widetilde{\psi}(V^{k})-(-\widetilde{\psi})^*(-\Gamma^{k})=\langle\Gamma^k,V^k\rangle$.
 Along with the expression of $\Theta_{\rho}$, we have
 \begin{align}\label{WL-Vk}
  \Theta_{\rho}(V^{k+1},\Gamma^{k},V^k)
  &\le\widetilde{f}(V^{k+1})+\langle\nabla\!\widetilde{f}(U^k),V^{k}-V^{k+1}\rangle
      +\rho\|V^k\|_F^2\!+\!\rho\widetilde{\psi}(V^k)\nonumber\\
  &\quad +\frac{\gamma\underline{L}}{2}\|V^{k+1}\!-\!V^k\|_F^2
         +\frac{L_k}{2}\|V^{k}\!-\!U^k\|_F^2-\frac{L_k}{2}\|V^{k+1}\!-\!U^k\|_F^2,\nonumber\\
  &\le \widetilde{f}(V^k)+\rho\|V^k\|_F^2\!+\!\rho\widetilde{\psi}(V^k)
  +\frac{\gamma\underline{L}}{2}\|V^{k+1}\!-\!V^k\|_F^2\nonumber\\
  &\quad +\frac{L_k+L_{\!\widetilde{f}}}{2}\|V^{k}\!-\!U^k\|_F^2
          -\frac{L_k\!-\!L_{\!\widetilde{f}}}{2}\|V^{k+1}\!-\!U^k\|_F^2
 \end{align}
 where the last inequality is using \eqref{wfrho} with $V=V^{k+1},Z=U^k$ and \eqref{-wfrho} with $V=V^{k},Z=U^k$.
 Notice that $\widetilde{\psi}(V^k)-\langle V^k,\Gamma^{k-1}\rangle\le (-\widetilde{\psi})^*(-\Gamma^{k-1})$.
 Together with the definition of $\Theta_{\rho}$ and $U^k=V^k+\beta_k(V^k\!-\!V^{k-1})$, it follows that
 \begin{align*}
  \Theta_{\rho}(V^{k+1},\Gamma^{k},V^k)
  &\le \Theta_{\rho}(V^{k},\Gamma^{k-1},V^{k-1})
       +\frac{\gamma\underline{L}}{2}\|V^{k+1}\!-\!V^k\|_F^2+\frac{L_k\!+\!L_{\!\widetilde{f}}}{2}\|V^{k}\!-\!U^k\|_F^2\nonumber\\
  &\quad -\frac{L_k\!-\!L_{\!\widetilde{f}}}{2}\|V^{k+1}\!-\!U^k\|_F^2
        -\frac{\gamma\underline{L}}{2}\|V^{k}\!-\!V^{k-1}\|_F^2\\
  &\le\Theta_{\rho}(V^{k},\Gamma^{k-1},V^{k-1})
    -\frac{L_k\!-\!\gamma\underline{L}\!-\!L_{\!\widetilde{f}}}{2}\|V^{k+1}\!-\!V^k\|_F^2\\
  &-\frac{\gamma\underline{L}\!-2L_{\!\widetilde{f}}\beta_k^2}{2}\|V^k\!-\!V^{k-1}\|_F^2
       +(L_k\!-\!L_{\!\widetilde{f}})\beta_k\langle V^{k+1}\!-\!V^k,V^{k}\!-\!V^{k-1}\rangle.
 \end{align*}
 Since
 \(
  |2(L_k\!-\!L_{\!\widetilde{f}})\beta_k\langle V^{k+1}\!-\!V^k,V^k\!-\!V^{k-1}\rangle|
   \le \mu(L_k\!-\!L_{\!\widetilde{f}})^2\|V^{k+1}\!-\!V^k\|_F^2+\frac{\beta_k^2}{\mu}\|V^k\!-\!V^{k-1}\|_F^2
  \)
 for any $\mu>0$, the following inequality holds for any $\mu>0$:
 \begin{align*}
  \Theta_{\rho}(V^{k+1},\Gamma^{k},V^k)
  &\le \Theta_{\rho}(V^{k},\Gamma^{k-1},V^{k-1})
     -\Big[\frac{\gamma\underline{L}-2L_{\!\widetilde{f}}\beta_k^2}{2}-\frac{\beta_k^2}{2\mu}\Big]\|V^k\!-\!V^{k-1}\|_F^2\\
  &\quad -\Big[\frac{L_k\!-\!\gamma\underline{L}\!-\!L_{\!\widetilde{f}}}{2}
  -\frac{(L_k\!-\!L_{\!\widetilde{f}})^2\mu}{2}\Big]\|V^{k+1}\!-\!V^k\|_F^2.
 \end{align*}
 By taking $\mu=\frac{L_k-\gamma\underline{L}\!-\!L_{\!\widetilde{f}}}{(L_k-L_{\!\widetilde{f}})^2}$,
 the desired result follows from the last inequality.

 \noindent
 {\bf(ii)-(iii)} The boundedness of $\{V^k\}$ is trivial.
 Since $\Gamma^{k}\in\widetilde{\partial}\psi(V^k)$, its boundedness
 is due to Remark \ref{remark-Alg2} (b). So, it suffices to prove part (iii).
 By part (i), the sequence $\{\Theta_{\rho}(V^{k},\Gamma^{k-1},V^{k-1})\}$
 is nonincreasing. Notice that $\Theta_{\rho}$ is proper lsc and level-bounded.
 From \cite[Theorem 1.9]{RW98}, $\Theta_{\rho}$ is bounded below. This means that
 the limit $\varpi^*:={\displaystyle\lim_{k\to\infty}}\Theta_{\rho}(V^k,\Gamma^{k},V^{k-1})$ exists.
  From $\nu_k\ge\frac{(\gamma\underline{L}-2L_{\!\widetilde{f}}\overline{\beta}^2)(L_0-L_{\!\widetilde{f}}
  -\gamma \underline{L})-(L_0-L_{\!\widetilde{f}})^2\overline{\beta}^2}{L_0-L_{\!\widetilde{f}}-\gamma\underline{L}}>0$
  and part (i), we obtain $\lim_{k\to\infty}\|V^k\!-\!V^{k-1}\|_F=0$.
  We next show that $\Theta_{\rho}\equiv\varpi^*$ on the set $\Delta_{\rho}$.
  Pick any $(\widehat{V},\widehat{\Gamma},\widehat{U})\in\Delta_{\rho}$.
  From part (ii), there exists $\mathcal{K}\subseteq\mathbb{N}$ such that
  $\lim_{\mathcal{K}\ni k\to\infty}(V^{k},\Gamma^{k-1},V^{k-1})=(\widehat{V},\widehat{\Gamma},\widehat{U})$.
  From the expression of $\Theta_{\rho}(V^{k},\Gamma^{k-1},V^{k-1})$,
  \begin{align*}
   \varpi^*
   &=\lim_{\mathcal{K}\ni k\to\infty}\big[\widetilde{f}(V^{k})+\rho\langle\Gamma^{k-1},V^{k}\rangle+\rho\|V^{k}\|_F^2
    +\rho(-\widetilde{\psi})^*(-\Gamma^{k-1})\big]\\
   &=\widetilde{f}(\widehat{V})+\rho\langle\widehat{\Gamma},\widehat{V}\rangle+\rho\|\widehat{V}\|_F^2
     +\rho(-\widetilde{\psi})^*(-\widehat{\Gamma})=\Theta_{\rho}(\widehat{V},\widehat{\Gamma},\widehat{V})
   =\Theta_{\rho}(\widehat{V},\widehat{\Gamma},\widehat{U}),
  \end{align*}
  where the second equality is by the continuity of $(-\widetilde{\psi})^*$ since
  $(-\widetilde{\psi})^*(U)=\frac{1}{4}\|U\|^2$ by \cite[Proposition 11.21]{RW98},
  the third one is using $\widehat{V}\in\mathcal{S}$
  implied by $\{V^{k}\}_{k\in\mathcal{K}}\subseteq\mathcal{S}$, and the last one is using
  $\widehat{V}=\widehat{U}$ implied by $\lim_{\mathcal{K}\ni k\to\infty}\|V^{k}\!-\!V^{k-1}\|_F=0$.

  \noindent
  {\bf(iv)} By invoking \cite[Exercise 8.8]{RW98},
  for any $(V,\Gamma,U)\!\in\!\mathcal{S}\times\mathbb{R}^{m\times p}\times\mathbb{R}^{m\times p}$,
  \begin{equation}\label{subdiff-Thetarho}
    \partial\Theta_{\rho}(V,\Gamma,U)
    =\left[\begin{matrix}
      \nabla\!\widetilde{f}(V)\!+\!2\rho V+\rho\Gamma+\gamma\underline{L}(V\!-\!U)+\mathcal{N}_{\mathcal{S}}(V)\\
      \rho V-\rho\partial(-\widetilde{\psi})^*(-\Gamma)\\
      \gamma\underline{L}(U\!-\!V)
    \end{matrix}\right].
  \end{equation}
  From the definition of $V^{k}$,
  \(
    0\in \nabla\!\widetilde{f}(U^{k-1})\!+\!\rho\Gamma^{k-1}+2\rho V^k+L_{k-1}(V^{k}\!-\!U^{k-1})
    +\mathcal{N}_{\mathcal{S}}(V^{k}).
  \)
  Since $\Gamma^{k-1}\in\partial\widetilde{\psi}(V^{k-1})\subseteq-\partial(-\widetilde{\psi})(V^{k-1})$,
  by invoking \cite[Theorem 23.5]{Roc70} we have $V^{k-1}\!\in\partial(-\widetilde{\psi})^*(-\Gamma^{k-1})$.
  By combining with the last equality yields, it follows that
  \[
    \left[\begin{matrix}
    \nabla\!\widetilde{f}(V^{k})-\!\nabla\!\widetilde{f}(U^{k-1})
    -\!L_{k-1}(V^{k}\!-\!U^{k-1})+\gamma\underline{L}(V^k\!-\!V^{k-1})\\
      \rho (V^k\!-\!V^{k-1})\\
      \gamma\underline{L}(V^{k-1}\!-\!V^k)
    \end{matrix}\right]\!\in\partial\Theta_{\rho}(V^k,\Gamma^{k-1},V^{k-1}).
  \]
  This along with $U^{k-1}=V^{k-1}+\beta_{k-1}(V^{k-1}\!-\!V^{k-2})$
  implies the desired result.  \qed
  \end{proof}

\bigskip
\noindent
{\large\bf Appendix B: Theoretical analysis of Algorithm \ref{Alg}}

 We first provide the convergence of Algorithm \ref{Alg-X}.
 From the Lipschitz continuity of $\nabla\!f$ on
 $\mathbb{B}_{\Omega}$, for every $X\!\in\mathbb{B}_{\Omega}$,
 \begin{subnumcases}{}
  \label{fdescent}
   f(X)\le f(Y)+\langle\nabla\!f(Y),X\!-\!Y\rangle+(L_{\!f}/2)\|X\!-\!Y\|_F^2;\\
   \label{-fdescent}
   -f(X)\le-f(Y)-\langle\nabla\!f(Y),X\!-\!Y\rangle+(L_{\!f}/2)\|X\!-\!Y\|_F^2,
 \end{subnumcases}
 where $L_{\!f}$ denotes the Lipschitz constant of $\nabla\!f$ in
 $\mathbb{B}_{\Omega}$.
 Algorithm \ref{Alg-X} is similar to the proximal DC algorithm proposed in \cite{LiuPong19},
 but the conclusion of \cite[Theorem 3.1]{LiuPong19} can not be directly applied to it
 since the convexity of $f$ is not required here. Inspired by the analysis technique in \cite{LiuPong19},
 we define the following potential function
 \begin{equation*}
   \Xi_{\rho}(X,W,Z)\!:=f(X)+\rho\langle I+W,X\rangle+\delta_{\Omega}(X)+\rho\delta_{\mathbb{B}}(-W)
   +\frac{L_{\!f}}{2}\|X\!-\!Z\|_F^2
 \end{equation*}
 associated to $\rho>0$, where $\mathbb{B}:=\{Z\in\mathbb{S}^p\ |\ \|Z\|_*\le 1\}$
 is the nuclear norm unit ball.
 \begin{proposition}\label{prop-Alg1}
  Let $\{(X^k,W^k)\}$ be the generated by Algorithm \ref{Alg-X}. Then,
  \begin{itemize}
   \item [(i)] \(
                 \Xi_{\rho}(X^{k+1},W^{k},X^{k})\!\le\!\Xi_{\rho}(X^{k},W^{k-1},X^{k-1})
                  \!-\frac{L_{\!f}-(L_k+L_{\!f})\beta_k^2}{2}\|X^{k}\!-\!X^{k-1}\|_F^2;
               \)

   \item [(ii)] the sequence $\{(X^k,W^k)\}$ is bounded, and consequently,
                the cluster point set of $\{(X^k,W^{k-1},\\ X^{k-1})\}$,
                denoted by $\Upsilon_{\!\rho}$, is nonempty and compact;

   \item[(iii)] the limit $\omega^*\!:=\lim_{k\to\infty}\Xi_{\rho}(X^k,W^{k-1},X^{k-1})$ exists
                whenever $\overline{\beta}<\!\sqrt{\frac{L_f}{L_0+L_f}}$,
                and moreover, $\Xi_{\rho}(X',W',Z')=\omega^*$ for every $(X',W',Z')\in\Upsilon_{\rho}$;

   \item[(iv)] for all $k\in\mathbb{N}$, with $\eta=\sqrt{9L_{\!f}^2\!+\!4L_0^2\!+\!\rho^2}$ it holds that
               \[
                 {\rm dist}(0,\partial\Xi_{\rho}(X^k,W^{k-1},X^{k-1}))
               \!\le\!\eta\big[\|X^k\!-\!X^{k-1}\|_F+\|X^{k-1}\!-\!X^{k-2}\|_F\big].
               \]
  \end{itemize}
 \end{proposition}
 \begin{proof}
  {\bf(i)} By the definition of $X^{k+1}$, the strong convexity of the objective function
  of \eqref{Xk-subprob}, and the feasibility of $X^{k}$ to the subproblem \eqref{Xk-subprob},
  it follows that
  \begin{align*}
   &\langle\nabla\!f(Y^k)\!+\!\rho(I\!+\!W^{k}),X^{k+1}\rangle +({L_k}/{2})\|X^{k+1}\!-\!Y^{k}\|_F^2\nonumber\\
   &\le \langle\nabla\!f(Y^k)\!+\!\rho(I\!+\!W^{k}),X^{k}\rangle+({L_k}/{2})\|X^{k}\!-\!Y^{k}\|_F^2
       -({L_k}/{2})\|X^{k+1}\!-\!X^k\|_F^2,
  \end{align*}
  which, after a suitable rearrangement, can be equivalently written as
  \begin{align}\label{key-ineq1}
   \rho\langle I\!+\!W^{k},X^{k+1}\rangle
   &\le \langle\nabla\!f(Y^k),X^k\!-\!X^{k+1}\rangle+\rho\langle I\!+\!W^{k},X^{k}\rangle
        +0.5L_k\|X^{k}\!-\!Y^{k}\|_F^2\nonumber\\
   &\quad -0.5L_k\|X^{k+1}\!-\!X^k\|_F^2-0.5L_k\|X^{k+1}\!-\!Y^{k}\|_F^2.
  \end{align}
  Since $W^{k}\in\partial\psi(X^{k})\subseteq-\partial(-\psi)(X^{k})$
  and the spectral function is the support of $\mathbb{B}$, we have $-W^{k}\in\mathbb{B}$
  and $-\langle W^k,X^k\rangle=\|X^k\|\ge-\langle W^{k-1},X^k\rangle$
  by \cite[Corollary 23.5.3]{Roc70}. Thus, for each $k\in\mathbb{N}$,
  $\delta_{\mathbb{B}}(W^k)=0$ and
  $\langle I\!+\!W^{k},X^k\rangle\le \langle I\!+\!W^{k-1},X^k\rangle$.
  Together with the definition of $\Xi_{\rho}$ and \eqref{key-ineq1}, it follows that
  \begin{align}\label{Lf-ineq}
   \Xi_{\rho}(X^{k+1},W^{k},X^{k})
   &\le f(X^{k+1})+\langle\nabla\!f(Y^k),X^k\!-\!X^{k+1}\rangle+\rho\langle I\!+\!W^{k-1},X^k\rangle\nonumber\\
   &\ +\frac{L_k}{2}\|X^{k}\!-\!Y^{k}\|_F^2-\frac{L_k}{2}\|X^{k+1}\!-\!Y^{k}\|_F^2
          -\frac{L_k\!-\!L_{\!f}}{2}\|X^{k+1}\!-\!X^{k}\|_F^2,\nonumber\\
   &\le f(X^k)+\rho\langle I\!+\!W^{k-1},X^k\rangle+\frac{L_k\!+L_{\!f}}{2}\|X^{k}\!-\!Y^{k}\|_F^2\\
   &\quad -\frac{L_{k}\!-\!L_{\!f}}{2}\|X^{k+1}\!-\!Y^{k}\|_F^2
   -\frac{L_k-L_{\!f}}{2}\|X^{k+1}\!-\!X^{k}\|_F^2,\nonumber
  \end{align}
  where the second inequality is obtained by using \eqref{fdescent} with $X=X^{k+1},Y=Y^k$,
  and \eqref{-fdescent} with $X=X^{k+1},Y=Y^k$.
  Now substituting $Y^k=X^k+\beta_k(X^k\!-\!X^{k-1})$ into the last inequality and
  using $L_k\ge L_f$ yields
  \begin{align*}
    \Xi_{\rho}(X^{k+1},W^{k},X^{k}) &\le\Xi_{\rho}(X^{k},W^{k-1},X^{k-1})
    -\frac{L_{\!f}\!-\!(L_k+L_{\!f})\beta_k^2}{2}\|X^k\!-\!X^{k-1}\|_F^2\\
   &\quad -\frac{L_{k}\!-\!L_{\!f}}{2}\|X^{k+1}\!-\!Y^{k}\|_F^2
   -\frac{L_k-L_{\!f}}{2}\|X^{k+1}\!-\!X^{k}\|_F^2\\
   &\le\Xi_{\rho}(X^{k},W^{k-1},X^{k-1})
    -\frac{L_{\!f}\!-\!(L_k+L_{\!f})\beta_k^2}{2}\|X^k\!-\!X^{k-1}\|_F^2.
  \end{align*}

  \noindent
  {\bf(ii)-(iii)} Part (ii) is immediate by noting that $\{X^k\}\subseteq\Omega$
  and $\{W^k\}\subseteq\mathbb{B}$. Next we prove part (iii).
  By part (i), the sequence $\{\Xi_{\rho}(X^{k},W^{k-1},X^{k-1})\}$ is nonincreasing.
  Notice that $\Xi_{\rho}$ is proper lsc and level-bounded. By \cite[Theorem 1.9]{RW98},
  it is bounded below. So, the limit $\omega^*$ is well defined.
  By part (i) and $L_{\!f}\!-\!(L_k+L_{\!f})\beta_k^2\ge L_{\!f}\!-\!(L_0+L_{\!f})\overline{\beta}^2>0$,
  we have $\lim_{k\to\infty}\|X^k\!-\!X^{k-1}\|_F=0$.
  We next show that $\Xi_{\rho}\equiv\omega^*$ on the set $\Upsilon_{\!\rho}$.
  Pick any $(\widehat{X},\widehat{W},\widehat{Z})\in\Upsilon_{\!\rho}$.
  By part (ii), there exists an index set $\mathcal{K}\subseteq\mathbb{N}$
  such that
  \(
    {\displaystyle\lim_{\mathcal{K}\ni k\to\infty}}(X^{k},W^{k-1},X^{k-1})=(\widehat{X},\widehat{W},\widehat{Z}).
  \)
  Along with the expression of $\Xi_{\rho}$,
  \begin{align*}
   \omega^*&=\lim_{\mathcal{K}\ni k\to\infty}\Xi_{\rho}(X^{k},W^{k-1},X^{k-1})
   =\lim_{\mathcal{K}\ni k\to\infty}\big[f(X^{k})+\rho\langle(I\!+\!W^{k-1}),X^{k}\rangle\big]\\
   &= f(\widehat{X})+\rho\langle I+\widehat{W},\widehat{X}\rangle
    =\Xi_{\rho}(\widehat{X},\widehat{W},\widehat{X})
   =\Xi_{\rho}(\widehat{X},\widehat{W},\widehat{Z}),
  \end{align*}
  where the second equality is since $\|X^k\!-\!X^{k-1}\|_F\to 0$
  and $\{(X^k,W^k)\}\subseteq\Omega\times\mathbb{B}$, and the last one
  is due to $\widehat{X}=\widehat{Z}$, implied by
  $\lim_{\mathcal{K}\ni k\to\infty}(X^{k},X^{k-1})=(\widehat{X},\widehat{Z})$.

  \noindent
  {\bf(iv)} Notice that $0\in\!\nabla\!f(Y^{k-1})+\rho(I+\!W^{k-1})+L_{k-1}(X^k-Y^{k-1})+\mathcal{N}_{\Omega}(X^k)$
  by the optimality condition of \eqref{Xk-subprob}.
  Recall that $W^{k-1}\in\partial\psi(X^{k-1})\subseteq-\partial(-\psi)(X^{k-1})$
  and the conjugate of the spectral function is $\delta_{\mathbb{B}}$. By \cite[Theorem 23.5]{Roc70},
  we have $X^{k-1}\in\partial\delta_{\mathbb{B}}(-W^{k-1})
  =\mathcal{N}_{\mathbb{B}}(-W^{k-1})$. Together with the expression of $\Xi_{\rho}$,
  we have
  \[
   \!\left[\begin{matrix}
     \nabla\!f(X^{k})\!-\!\nabla\!f(Y^{k-1})\!+\!L_{\!f}(X^k\!-\!X^{k-1})\!-\!L_{k-1}(X^{k}\!-\!Y^{k-1})\\
      \rho (X^k\!-\!X^{k-1})\\
      L_{\!f}(X^{k-1}\!-\!X^k)
    \end{matrix}\right]\!\in\partial\Xi_{\rho}(X^k,W^{k-1},X^{k-1}).
  \]
  This implies that the desired inequality holds. Thus, we complete the proof.
 \end{proof}
 \begin{remark}\label{remark-XkWk}
 {\bf(a)} When $f$ is convex, the coefficient $L_{\!f}$ appearing in \eqref{Lf-ineq}
  can be removed. So, the restriction on $\overline{\beta}$ in part (iii)
  can be improved as $\overline{\beta}<\!\sqrt{{L_{\!f}}/{L_0}}$. This coincides with
  the requirement of \cite[Proposition 3.1]{LiuPong19} for the convex $f$.

  \noindent
  {\bf(b)} Let $(\widehat{X},\widehat{W})$ be an accumulation point
  of $\{(X^k,W^k)\}$. By the outer semicontinuity of $\mathcal{N}_{\Omega}$
  and $\partial\psi$, we have $\widehat{W}\in\partial\psi(\widehat{X})$
  and $0\in \nabla\!f(\widehat{X})+\rho(I+\widehat{W})+\mathcal{N}_{\Omega}(\widehat{X})$
  which by the expression of $\partial\Xi_{\rho}$ and Definition \ref{Spoint-def1} implies that
  $\Pi_1(\Upsilon_{\!\rho})\subseteq\Pi_1({\rm crit}\,\Xi_{\!\rho})\subseteq\overline{\Omega}_{\rho}$,
  where $\Pi_1(\Upsilon_{\!\rho})\!:=\{Z\in\mathbb{S}^p\,|\,\exists W\ {\rm s.t.}\
  (Z,W,Z)\in\Upsilon_{\!\rho}\}$.
 \end{remark}

  By \cite[Section 4.3]{Attouch10}, the indicator functions $\delta_{\Omega}$
  and $\delta_{\mathbb{B}}$ are semialgebraic, which implies that $\Xi_{\rho}$
  is a KL function (see \cite{Attouch10} for the detail).
  By using Proposition \ref{prop-Alg1} and the same arguments as those
  for \cite[Theorem 3.1]{LiuPong19} (see also \cite[Theorem 3.1]{Attouch10}),
  we obtain the following conclusion.
 \begin{theorem}\label{theorem-Alg1}
  Let $\{(X^k,W^k)\}$ be the sequence generated by Algorithm \ref{Alg-X} from
  $X^0=X^{l}$ with $\overline{\beta}<\!\sqrt{\frac{L_{\!f}}{L_0+L_{\!f}}}$ for
  solving \eqref{epenalty} associated to $\rho_l$.
  Then, the sequence $\{X^k\}$ is convergent, and its limit is a critical point
  of \eqref{epenalty} associated to $\rho_l$. If this limit is rank-one, it
  is also a local optimal solution of \eqref{DC-BPP}.
 \end{theorem}

  We have provided a convergent algorithm to seek a critical point of
  the subproblems in Algorithm \ref{Alg}. Next we focus on the stopping criterion
  of Algorithm \ref{Alg} which aims to seek an approximate rank-one critical point.
  When this criterion occurs at some $l<l_{\rm max}$, we say that Algorithm \ref{Alg}
  exits normally. The following proposition states that under a certain condition
  Algorithm \ref{Alg} can exit normally.
 \begin{proposition}\label{defined-Alg1}
  Fix an arbitrary integer $l\ge 0$. Suppose that $X^l\in\Omega$ satisfies
  $\langle I,X^l\rangle-\|X^l\|\le c_0$ for some $c_0\in(0,1)$.
  Then for any given $\varepsilon\in\!(0,c_0]$, the following results
  hold for the sequence $\{X^k\}$ generated by Algorithm \ref{Alg-X}
  from $X^0=X^l$ with $\rho_{l}\ge\max\big\{\frac{2\varpi}{(1-\sqrt{1-0.5p^{-1}\varepsilon})\sqrt{1-c_0}},\frac{2\varpi}{\varepsilon}\big\}$
  for $\varpi\!=6.5(L_{\!f}\!+\!L_0)p^2\!+\!2p\|\nabla\!f(I)\|_F$:
  \begin{itemize}
  \item [(i)] for each integer $k\ge 0$ with $\frac{\varepsilon}{2}\le \langle I,X^k\rangle-\|X^k\|\le c_0$,
              \begin{equation}\label{snormX0-increase}
               \|X^{k+1}\|\ge\|X^k\|+0.5(1-\!\sqrt{1-0.5p^{-1}\varepsilon})\sqrt{1-c_0};
               \end{equation}
  \item [(ii)] there exists $\overline{k}\le\frac{2(c_0-\varepsilon)}{(1-\sqrt{1-0.5p^{-1}\varepsilon})\sqrt{1-c_0}}+1$
               such that $\langle I,X^{k}\rangle-\|X^{k}\|\le\varepsilon$ for all $k\ge\overline{k}$.
  \end{itemize}
 \end{proposition}
 \begin{proof}
  {\bf(i)} For each $k\in\mathbb{N}$, from the definition of $X^{k+1}$,
  for any $X\in\Omega$ we have
  \begin{align}\label{temp-C}
   \rho_l\langle W^{k},X\!-\!X^{k+1}\rangle
   &\le\langle\nabla\!f(Y^{k}),X\!-\!X^{k+1}\rangle
   +\frac{L_k}{2}\|X\!-\!Y^{k}\|_F^2\!-\!\frac{L_k}{2}\|X^{k+1}\!-\!Y^{k}\|_F^2\nonumber\\
   &\le\langle \nabla\!f(Y^{k}),X\!-\!X^{k+1}\rangle+\frac{L_k}{2}\big[\|X\|_F^2+2\|X^{k+1}\!-X\|_F\|Y^k\|_F\big]\nonumber\\
   &\le\langle \nabla\!f(Y^{k})\!-\!\nabla\!f(I)\! +\!\nabla\!f(I),X\!-\!X^{k+1}\rangle+6.5L_0p^2\nonumber\\
   &\le (L_{\!f}\|Y^{k}\!-\!I\|_F\!+\!\|\nabla\!f(I)\|_F)\|X\!-\!X^{k+1}\|_F\!+\!6.5L_0p^2\nonumber\\
   &\le 2p(3L_{\!f}p+\|\nabla f(I)\|_F)\!+\!6.5L_0p^2\le\varpi
  \end{align}
  where the third inequality is by $L_k\le L_0$ for all $k\in\mathbb{N}$
  and $\|X\|_F\le p$ for $X\in\Omega$. So
  \begin{align}\label{temp-WX}
   \langle W^{k},X\rangle
   \le\frac{\varpi}{\rho_l}+\langle W^{k},X^{k+1}\rangle
   \le\frac{\varpi}{\rho_l}+\|W^k\|_*\|X^{k+1}\|=\frac{\varpi}{\rho_l}+\|X^{k+1}\|.
  \end{align}
  Let $X^k$ have the eigenvalue decomposition $U{\rm Diag}(\lambda(X))U^\mathbb{T}$
  with $U\!=[u_1 \cdots u_p]\in\mathbb{O}^p$.
  Since $\langle I,X^k\rangle-\|X^k\|\le c_0<1$ and ${\rm Diag}(X^{k})=e$,
  for every $j\in\{1,2,\ldots,p\}$,
  \begin{align}\label{lambdaXu}
   \lambda_1(X^k)u_{1j}^2=1-{\textstyle\sum_{i=2}^p}\lambda_i(X^k)u_{ij}^2
   \ge 1-{\textstyle\sum_{i=2}^p}\lambda_i(X^k)\ge 1-c_0>0.
  \end{align}
  Take $\widehat{X}=\lambda_1(X^k)\widehat{u}_1\widehat{u}_1^\mathbb{T}$ with
  $\widehat{u}_{1j}=\frac{u_{1j}}{\sqrt{\|X^k\|u_{1j}^2}}$ for $j=1,\ldots,p$.
  It is easy to check that $\widehat{X}\in\Omega$. Now using \eqref{temp-WX}
  with $X=\widehat{X}$ and recalling that $W^k=u_1u_1^{\mathbb{T}}$, we obtain
  \begin{align}\label{temp-snromX}
   &\frac{\varpi}{\rho_l}+\|X^{k+1}\|\ge\langle W^k,\widehat{X}\rangle
   =\|X^k\|(u_1^{\mathbb{T}}\widehat{u}_1)^2
   =\Big(\sqrt{u_{11}^2}+\cdots+\sqrt{u_{1p}^2}\Big)^2\nonumber\\
   =&\|X^k\|+\!\Big[\sqrt{u_{11}^2}+\cdots+\sqrt{u_{1p}^2}-\!\sqrt{\|X^k\|}\Big]
   \Big[\sqrt{u_{11}^2}+\cdots+\sqrt{u_{1p}^2}+\!\sqrt{\|X^k\|}\Big]\nonumber\\
   \ge&\|X^k\|+\Big[\sqrt{u_{11}^2}+\cdots+\sqrt{u_{1p}^2}-\!\sqrt{\|X^k\|}\Big]\sqrt{\|X^k\|}\nonumber\\
   =&\|X^k\|+\!\Big[\sqrt{u_{11}^2}+\cdots+\sqrt{u_{1p}^2}-\sqrt{\|X^k\|}(u_{11}^2+\cdots+u_{1p}^2)\Big]\sqrt{\|X^k\|},
  \end{align}
  where the second inequality is using $\sqrt{u_{11}^2}+\cdots+\sqrt{u_{1p}^2}\ge\!\sqrt{\|X^k\|}$
  implied by \eqref{lambdaXu}, and the last equality is by $\sum_{j=1}^pu_{1j}^2=1$. Since $\|u_1\|=1$,
  there exists $\widehat{j}\in\{1,\ldots,p\}$ such that $u_{1\widehat{j}}^2\le\frac{1}{p}$.
  Note that $1-\sqrt{\|X^k\|u_{1j}^2}\ge0$ for all $j=1,\ldots,p$.
  From \eqref{temp-snromX},
  \begin{align*}
   \rho_l^{-1}\varpi+\|X^{k+1}\|&
   \ge\|X^k\|+\Big[\sqrt{u_{1\widehat{j}}^2}-\!\sqrt{\|X^k\|}u_{1\widehat{j}}^2\Big]\sqrt{\|X^k\|}\\
   &=\|X^k\|+\big(1-\!\sqrt{\|X^k\|u_{1\widehat{j}}^2}\big)\sqrt{\|X^k\|u_{1\widehat{j}}^2}\\
   &\ge\|X^k\|+\big(1-\!\sqrt{\|X^k\|/p}\big)\sqrt{\|X^k\|u_{1\widehat{j}}^2}\\
   &\ge\|X^k\|+(1\!-\!\sqrt{1-0.5p^{-1}\varepsilon})\sqrt{1-c_0}
  \end{align*}
  where the second inequality is due to $u_{1\widehat{j}}^2\le\frac{1}{p}$,
  and the last one is using $p-\|X^k\|\ge\frac{\varepsilon}{2}$ and \eqref{lambdaXu}.
  Together with $\rho_l\ge\frac{2C}{(1-\sqrt{1-0.5p^{-1}\varepsilon})\sqrt{1-c_0}}$,
  we get the desired result.

  \noindent
  {\bf(ii)} Since the proof is similar to that of Proposition \ref{defined-Algfactor1} (ii),
  we here delete it.
 \end{proof}

 Unlike for Algorithm \ref{Alg-factor}, now we can not provide a suitable condition
 to ensure that some $X^l\in\Omega$ with $\langle I,X^l\rangle-\|X^l\|\le\!c_0$
 occurs, and then Algorithm \ref{Alg} exists normally. We leave this question
 for a research topic. To close this section, we show that the rank-one projection of
 its normal output is an approximately feasible solution of \eqref{UBPP},
 and provide an upper estimation of its objective value to the optimal one.
 \begin{theorem}\label{obj-bound1}
  Let $\upsilon^*$ denote the optimal value of \eqref{UBPP} and $X^{l_{\!f}}$ be the normal
  output of Algorithm \ref{Alg}. For each $l\ge 0$, let $\{(X^{l,k},W^{l,k})\}$
  be the sequence generated by Algorithm \ref{Alg-X} with $X^{l,0}=X^{l}$
  and $\beta_k\equiv 0$. If there exists $l^*\in\{0,1,\ldots,l_{\!f}\}$ such that
  $f(X^{l^*})\le \upsilon^*$, then the following inequalities hold
  with $r^*\!={\rm rank}(X^{l^*})$:
  \begin{subequations}
   \begin{align}\label{objbound-ineq1}
    f(x^{l_{\!f}}(x^{l_{\!f}})^{\mathbb{T}})-\upsilon^*
    \le \rho_{l_{\!f}}\|X^{l_{\!f}}\|-\rho_{l^*}p/r^*+{\textstyle\sum_{j=l^*}^{l_{\!f}-1}}(\rho_{j}-\rho_{j+1})\|X^{j+1}\|
        +\alpha_{\!f}\epsilon,\\
    \label{objbound-ineq2}
    \|x^{l_{\!f}}\circ x^{l_{\!f}}-e\|\le\epsilon\quad{\rm with}\ x^{l_{\!f}}=\|X^{l_{\!f}}\|^{1/2}P_1
    \ {\rm for}\ P\in\!\mathbb{O}(X^{l_{\!f}}).
    \qquad\qquad
   \end{align}
  \end{subequations}
  \end{theorem}
 \begin{proof}
  Fix any $l\in\{0,1,\ldots,l_{\!f}\}$. For each $k\ge 0$,
  from $\beta_k\equiv 0$ and \eqref{key-ineq1},
  \[
   \langle \nabla f(X^{l,k})+\rho_l(I+W^{l,k}),X^{l,k+1}\!-\!X^{l,k}\rangle
   +L_k\|X^{l,k+1}\!-\!X^{l,k}\|_F^2\le 0.
  \]
  Since $f(X^{l,k+1})\le f(X^{l,k})+\langle\nabla\!f(X^{l,k}),X^{l,k+1}\!-\!X^{l,k}\rangle
  +\frac{L_f}{2}\|X^{l,k+1}\!-\!X^{l,k}\|_F^2$ by using \eqref{fdescent} with $X=X^{l,k+1},Y=X^{l,k}$,
  from the last inequality it follows that
  \[
    f(X^{l,k+1})-f(X^{l,k})+\rho_l\langle(I+W^{l,k}),X^{l,k+1}\!-\!X^{l,k}\rangle
    \le\frac{L_{\!f}-2L_k}{2}\|X^{l,k+1}\!-\!X^{l,k}\|_F^2.
  \]
  Notice that
  \(
    -\|X^{l,k+1}\|+\|X^{l,k}\|\le\langle W^{l,k},X^{l,k+1}\!-\!X^{l,k}\rangle
  \)
  and $\langle I,X^{l,k+1}\rangle=\langle I,X^{l,k}\rangle=p$
  by $X^{l,k+1},X^{l,k}\in\Omega$, and $L_k\ge L_{\!f}$. Then,
  \(
    f(X^{l,k+1})\!-\!\rho_l\|X^{l,k+1}\|
   \le f(X^{l,k})-\rho_l\|X^{l,k}\|.
  \)
  From this recursion formula, it immediately follows that
  \[
    f(X^{l,k+1})+\rho_l\|X^{l,k+1}\|
    \le\cdots\le f(X^{l,0})-\rho_l\|X^{l,0}\|=f(X^{l})-\rho_{l}\|X^{l}\|.
  \]
  By Theorem \ref{theorem-Alg1}, the sequence $\{X^{l,k}\}$ is convergent as $k\to\infty$.
  Let $X^{l,*}$ denote its limit. Then $X^{l+1}=X^{l,*}$.
  From the last inequality, for each $l\in\{0,1,\ldots,l_{\!f}\}$,
  \begin{equation}\label{lineqA1}
    f(X^{l+1})-\rho_l\|X^{l+1}\|\le f(X^{l})-\rho_{l}\|X^{l}\|.
  \end{equation}
  Notice that $X^{l_{\!f}}=\sum_{i=1}^p\lambda_i(X^{l_{\!f}})P_iP_i^{\mathbb{T}}\in\Omega$
  where $P_i$ denotes the $i$th column of $P$.
  By the Lipschitz continuity of $f$ relative to $\Omega$ with modulus $\alpha_{\!f}$,
  it follows that
  \begin{align}\label{xlf-ineqA}
  f(X^{l_{\!f}})&=f\big({\textstyle\sum_{i=1}^p}\lambda_i(X^{l_{\!f}})P_iP_i^{\mathbb{T}}\big)
  =f\big(\lambda_1(X^{l_{\!f}})P_1P_1^{\mathbb{T}}+{\textstyle\sum_{i=2}^p}\lambda_i(X^{l_{\!f}})P_iP_i^{\mathbb{T}}\big)
    \nonumber\\
  &\ge f(x^{l_{\!f}}(x^{l_{\!f}})^{\mathbb{T}})-
      \alpha_{\!f}\|{\textstyle\sum_{i=2}^p}\lambda_i(X^{l_{\!f}})P_iP_i^{\mathbb{T}}\|_F
   \ge f(x^{l_{\!f}}(x^{l_{\!f}})^{\mathbb{T}})-\alpha_{\!f}\epsilon.
  \end{align}
  In addition, adding $(\rho_{l}-\rho_{l+1})\|X^{l+1}\|$ to
  the both sides of \eqref{lineqA1} yields that
  \[
    f(X^{l+1})-\rho_{l+1}\|X^{l+1}\|
    \le f(X^{l})-\rho_{l}\|X^{l}\|+(\rho_{l}\!-\!\rho_{l+1})\|X^{l+1}\|.
  \]
  Thus,
  \(
    f(X^{l_{\!f}})-\rho_{l_{\!f}}\|X^{l_{\!f}}\|
    \le f(X^{l^*})-\rho_{l^*}\|X^{l^*}\|+{\textstyle\sum_{j=l^*}^{l_{\!f}-1}}(\rho_{j}-\rho_{j+1})\|X^{j+1}\|.
  \)
  Combining this inequality with \eqref{xlf-ineqA} and
  noting that $\|X^{l^*}\|\ge p/r^*$ yields \eqref{objbound-ineq1}.
  Since ${\rm diag}(X^{l_{\!f}})=e$, we have
  $\|x^{l_{\!f}}\circ x^{l_{\!f}}-e\|=\|\sum_{i=2}^p\lambda_i(X^{l_{\!f}})P_i\circ P_i\|\le\epsilon$.
 \end{proof}

 By Theorem \ref{obj-bound1}, when Algorithm \ref{Alg} exits normally at the $l_{\!f}$th step
 and $f(X^{l^*})\le \upsilon^*$ for some $l^*\in\{0,1,\ldots,l_{\!f}\}$, the rank-one projection
 $x^{l_{\!f}}$ of $X^{l_{\!f}}$ delivers an approximate feasible solution of problem \eqref{UBPP},
 and the difference between its objective value and the optimal value of \eqref{UBPP} is
 upper bounded by the right hand side of \eqref{objbound-ineq1}, which becomes less
 if $l^*$ is closer to $l_{\!f}-1$ or the rank of $X^{l^*}$ is close to $1$.
 Clearly, if there exists $l^*\in\{0,1,\ldots,l_f\}$
 such that $f(X^{l^*})$ is close to the optimal value of \eqref{DC-BPP}
 without the DC constraint, it is more likely for
 $f(X^{l^*})\le \upsilon^*$ to hold.

\end{document}